\documentclass[11pt, a4paper, final]{amsart}
\usepackage[T1]{fontenc}
\usepackage{mathtools}
\usepackage[activate={true,nocompatibility},final,tracking=true,kerning=true,spacing=true,stretch=10,shrink=10]{microtype}
\microtypecontext{spacing=nonfrench}

\usepackage{amsmath}
\usepackage{amsthm}
\usepackage{XCharter}
\usepackage[charter,expert, greekuppercase=italicized, greekfamily=didot]{mathdesign}
\usepackage{mathtools}
\usepackage{enumitem}
\usepackage[utf8]{inputenc}
\usepackage{tikz-cd}
\usepackage{tikz}

\usepackage{etoolbox}

\usepackage{xcolor}
\usepackage[colorlinks=true]{hyperref}

\usepackage[notref, notcite]{showkeys}
\usepackage[cmtip,arrow]{xy}
\usepackage[backend=biber,
url=false,
isbn=false,
backref=true,
citestyle=alphabetic,
bibstyle=alphabetic,
autocite=inline,
maxnames=99,
minalphanames=4,
maxalphanames=4,
sorting=nyt,]{biblatex}
\DeclareMathOperator{\Aut}{Aut}
\DeclareMathOperator{\End}{End}
\DeclareMathOperator{\Emb}{Emb}

\DeclareMathOperator{\Age}{Age}

\DeclareMathOperator{\GL}{GL}

\DeclareMathOperator{\dom}{{dom}}

\DeclareMathOperator{\qftp}{{qftp}}

\newcommand{\forkindep}[1][]{%
  \mathrel{
    \mathop{
      \vcenter{
        \hbox{\oalign{\noalign{\kern-.3ex}\hfil$\vert$\hfil\cr
              \noalign{\kern-.7ex}
              $\smile$\cr\noalign{\kern-.3ex}}}
      }
    }\displaylimits_{#1}
  }
}

\newcommand{\abs}[1]{\lvert{#1}\rvert}

\newcommand{\bN}{\mathbf N}

\newcommand{\bZ}{\mathbf Z}
\newcommand{\bQ}{\mathbf Q}

\newcommand{\cK}{\mathcal K}
\newcommand{\restr}{\mathord{\upharpoonright}}

\DeclareMathOperator{\Th}{{Th}}

\DeclareMathOperator{\id}{{id}}

\newtheorem*{mainthm}{Main Theorem}
\newtheorem{theorem}{Theorem}
\numberwithin{theorem}{section}
\newtheorem{lemma}[theorem]{Lemma}

\newtheorem{fact}[theorem]{Fact}
\newtheorem{proposition}[theorem]{Proposition}

\newtheorem{question}[theorem]{Question}
\newtheorem{corollary}[theorem]{Corollary}

\newtheorem{clm}{Claim}
\newtheorem*{clm*}{Claim}

\theoremstyle{definition}
\newtheorem{definition}[theorem]{Definition}

\newtheorem{example}[theorem]{Example}

\theoremstyle{remark}
\newtheorem{remark}[theorem]{Remark}

\AtEndEnvironment{proof}{\setcounter{clm}{0}}
\newenvironment{clmproof}[1][\proofname]{\proof[#1]}{\endproof}

\title{Inner ultrahomogeneous groups}
\subjclass[2020]{03C45, 03C15, 03E15, 20F05, 20A15}
\keywords{ample generic automorphisms, Hall's universal group, recursively presentable group, finitely presentable group, Fraïssé limit}

\author{Tomasz Rzepecki}
\thanks{The author was supported by the Institute of Mathematics, Czech Academy of Sciences (RVO 67985840) and the EXPRO 20-31529X grant (Czech Science Foundation).}
\address{Institute of Mathematics, Czech Academy of Sciences, Czech Republic and Uniwersytet Wrocławski, Poland}
\email{tomasz.rzepecki@math.uni.wroc.pl\\ORCID: \href{https://orcid.org/0000-0001-9786-164}{\includegraphics[height=\fontcharht\font`\B]{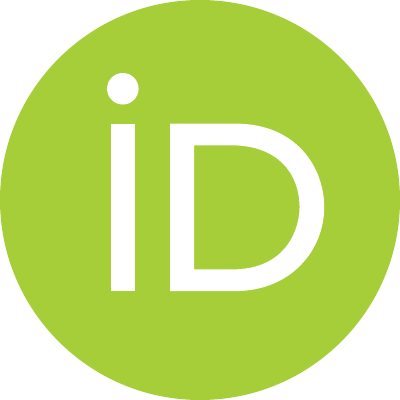}0000-0001-9786-1648}}

\begin{document}
    \begin{abstract}
        We define and study the class of inner ultrahomogeneous groups, which includes Hall's universal group and the universal locally recursively presentable group. We provide simple criteria for ample generic automorphisms, straight maximality, uniform simplicity and divisibility (all of which apply to both Hall's universal group and the universal locally recursively presentable group). We show that such groups of infinite exponent are not $\aleph_0$-saturated, their theories are not small, not rosy and have TP$_2$+SOP+IP$_n$ for all $n$. This strengthens and generalises known results about ample generic automorphisms and unstability of Hall's universal group. We also show that the exponents of finite exponent inner ultrahomogeneous groups are uniformly bounded, and we provide a series of examples of inner ultrahomogeneous groups.
    \end{abstract}
    \maketitle
    \section{Introduction}
    Hall's universal group $\Gamma_F$ was introduced by Phillip Hall in \cite{Hal59}. He showed that it is universal for the class of countable locally finite groups and homogeneous in a rather strong sense. While Hall did not explicitly use the Fraïssé construction (developed by Fraïssé around the same time), he essentially showed that the class of finite groups is a Fraïssé class and described its limit.

    $\Gamma_F$ has a rather curious ultrahomogeneity property, which we call inner ultrahomogeneity: every finite partial automorphism can be extended to not only an automorphism of the whole group, but to an inner automorphism (see Definition~\ref{dfn:inner_uh}). This property is related to a strengthening of the extension property of partial automorphisms (EPPA, also called Hrushovski's property) which we call here inner EPPA (see Definition~\ref{dfn:inner_eppa}) and which, due to Hall, holds for the class of finite groups.

    Another inner ultrahomogeneous group is $\Gamma_R$, the generic locally recursively presentable group (the Fraïssé limit of the class of finitely generated, recursively presentable groups, or equivalently, of finitely presentable groups). This group, while very natural, seems to not have been studied before, although its existence is folklore. We will describe it in more detail in Section~\ref{sec:examples}.

    Independently in \cite{Song19}, \cite{Sin17} and implicitly \cite{Iva06} (and also in unpublished work by Grebík and Geschke), the authors showed that Hall's universal group has ample generic automorphisms. This is a dynamical property of the automorphism group as a permutation group which is of interest in both model theory and dynamics.

    \cite{Song19} also observed that the group is unstable (in the sense of model theory). In fact, the proof of \cite[Proposition 3.1]{SU06} easily implies that many groups are SOP$_3$ (which is a stronger property than unstability), including $\Gamma_R$ and $\Gamma_F$. In contrast, \cite[Theorem 3.2]{SU06} shows that a suitably chosen ``universal'' group is NSOP$_4$ in the Robinson sense, which is a tameness property.

    It turns that inner ultrahomogeneity, along with some very natural hypotheses about the underlying Fraïssé class, suffices to prove many properties of Hall's group (both known before and new) in greater generality, for instance divisibility, uniform (group-theoretic) simplicity, ample generic automorphisms, as well as (first-order) unstability, independence property, TP$_2$, and straight maximality, using relatively simple formulas related to commuting. This implies in particular that many groups, including those considered in \cite{SU06}, are actually extremely wild with respect to full first order logic.

    We also observe that some of these groups, including $\Gamma_R$, have a natural stationary independence relation (see Proposition~\ref{prop:free_stationary_indep}), although this is not strictly related to inner ultrahomogeneity.

    The main results are summarised in \hyperref[mainthm]{Main Theorem} below, from which we deduce in Section~\ref{sec:examples} the aforementioned properties of Hall's universal group $\Gamma_F$, the group $\Gamma_R$ and other inner ultrahomogeneous groups.

    \begin{mainthm}
        \phantomsection
        \label{mainthm}
        For every inner ultrahomogeneous group $\Gamma$:
        \begin{enumerate}
            \item
            \label{it:mainthm:infinite}
            If $\Gamma$ has more than $6$ elements, then it is infinite and not $\aleph_0$-saturated (in particular, it is not $\aleph_0$-categorical),
            \item
            \label{it:mainthm:simplicity}
            If $\Gamma$ is either torsion and embeds all finite groups, or $\Age(\Gamma)$ is closed under $*\bZ$, or $\Gamma$ is torsion-free, then it is uniformly simple.
            \item
            \label{it:mainthm:ample_gen}
            If $\Age(\Gamma)$ is closed under either $\times$ or under $\times\bZ$ (which holds in particular if it is closed under free products or finitary HNN-extensions) and $\Gamma$ is countable, then it has ample generic automorphisms.
            \item
            \label{it:mainthm:centre}
            If the centre of $\Gamma$ is nontrivial, then its only nonidentity element is the unique element of $\Gamma$ of order $2$. In this case, $\Gamma$ has $2$ elements or is not torsion.
            \item
            \label{it:mainthm:finite_exp}
            If $\Gamma$ has finite exponent, then it has a finitely generated subgroup $A$ such that $\Aut(\Gamma/A)$ is trivial (i.e.\ $\Aut(\Gamma)$ is discrete) and every automorphism of $\Gamma$ is inner; in particular, if $\lvert\Gamma\rvert\neq 2$, then $\Gamma\cong \Aut(\Gamma)$.
            \item
            \label{it:mainthm:fin_exp_char}
            $\Gamma$ does not have finite exponent if and only if it contains a copy of $(\bZ/2\bZ)^6$ or an abelian subgroup with at least $2^{100}$ elements.
            \item
            \label{it:mainthm:not_finite_exp}
            If $\Gamma$ is not of finite exponent, then:
            \begin{itemize}
                \item
                if it is torsion, then it embeds every countable torsion abelian group,
                \item
                if it is not torsion, then it embeds every countable torsion-free abelian group and every countable free group,
                \item
                it does not admit elimination of quantifiers,
                \item
                it has the strict order property and the tree property of the second kind (in particular, it is unstable and has the independence property), as well as the $n$-independence property for all $n$,
                \item
                the theory of $\Gamma$ is not small.
            \end{itemize}
            \item
            \label{it:mainthm:def_subg}
            If $\Age(\Gamma)$ has disjoint amalgamation (e.g.\ if it is closed under amalgamated free products or it is the class of finite groups), then for every finite $A_0\subseteq A$, $C^2(A_0)=\langle A_0\rangle$. In particular, for each $n$, the family of $n$-generated subgroups of $\Gamma$ is uniformly definable.
            \item
            \label{it:mainthm:str_max}
            If $\Gamma$ has elements of all finite orders (in particular, if it is torsion and not of finite exponent), then it is straightly maximal.
        \end{enumerate}
    \end{mainthm}
    \begin{proof}
        In \eqref{it:mainthm:infinite}, that $\Gamma$ has to be infinite if it has more than $6$ elements is true even if we only assume that every automorphism of a finite subgroup of $\Gamma$ extends to an inner automorphism, see \cite{HoltMO}. The fact that it is not $\aleph_0$-saturated and its theory is not $\aleph_0$-categorical is Corollary~\ref{cor:not_saturated} and Corollary~\ref{cor:not_categorical}.

        \eqref{it:mainthm:simplicity} follows from Proposition~\ref{prop:finite_groups_simple}, Proposition~\ref{prop:free_prod_simple} and Remark~\ref{rem:conj_order}.

        \eqref{it:mainthm:ample_gen} is Theorem~\ref{thm:ample_generics}.

        \eqref{it:mainthm:centre} is Proposition~\ref{prop:center_at_most_2} and Corollary~\ref{cor:torsion_trivial_centre}.

        \eqref{it:mainthm:finite_exp} is Proposition~\ref{prop:fin_exp_almost_fg} and Remark~\ref{rem:fin_exp_trivial_centre}.

        \eqref{it:mainthm:fin_exp_char} follows from Theorem~\ref{thm:char_of_finite_exp}.

        \eqref{it:mainthm:not_finite_exp} follows from Corollary~\ref{cor:trichotomy}, Proposition~\ref{prop:no_qe}, Theorem~\ref{thm:inf_exp_untame} and Proposition~\ref{prop:not_small}.

        \eqref{it:mainthm:def_subg} is Proposition~\ref{prop:subgroup_definability}  (and Fact~\ref{fct:finite_amalg} for the locally finite case).

        \eqref{it:mainthm:str_max} follows from Theorem~\ref{thm:straight_maximality} and \eqref{it:mainthm:not_finite_exp}
    \end{proof}

    \subsection*{Structure}
    In the second section, we very briefly recall some of the elementary notions and fundamental observations used in the paper.
    In the third section, we introduce the notions of inner ultrahomogeneity and inner EPPA, and prove some initial lemmata which will be of importance in later sections.
    The fourth section is dedicated to various algebraic properties of inner ultrahomogeneous groups. In particular, it contains a characterisation of finite exponent inner ultrahomogeneous groups, and a trichotomy related to the existence of various abelian subgroups (see Corollary~\ref{cor:trichotomy}). It also discusses simplicity, divisibility and normal subgroups, particularly the centre.
    The fifth section is devoted to model-theoretic properties of inner ultrahomogeneous groups, and essentially showing that under some weak assumptions they lie on the wild side of all sorts of dividing lines in model theory.
    The sixth section is about providing criteria showing that many groups of this sort have ample generic automorphisms.
    Finally, the seventh section contains various examples of inner ultrahomogeneous groups and a discussion of their properties, mostly deduced via \hyperref[mainthm]{Main Theorem}.

    \section{Preliminaries}
    \subsection{Some conventions and notation}
    Given a group $G$ and $A\subseteq G$, we write $C(A)$ for the centraliser of $A$ and $C^2(A)$ for the double centraliser, i.e.\ $C(C(A))$; when $A\leq G$, we will write $N(A)$ for the normaliser (the background group will always be clear from the context). We will write $Z(G)$ for the centre of $G$.

    Given $g,h\in G$, we will use the exponentiation convention for conjugation: $g^h=h^{-1}gh$, so that it is written as a right action, i.e.\ $(g^{h_1})^{h_2}=g^{h_1h_2}$ (although in all but a handful of places in the paper, the statements are essentially agnostic about the direction of conjugation).

    This gives a homomorphism $G\to \Aut(G)$ given by $h\mapsto (g\mapsto g^{h^{-1}})$, and for $A\leq G$, this induces a homomorphism $N(A)\to \Aut(A)$ whose kernel is $C(A)$.

    We will sometimes implicitly refer to the conjugation action and say that e.g.\ $h$ \emph{fixes} $g$ if it commutes with it (so that $g^h=g$), it \emph{doubles} $g$ if $g^h=g^2$ and it \emph{inverts} $g$ if $g^h=g^{-1}$.

    For a group $G$, we will write $G^{\oplus \bN}$ for the direct sum/restricted product of countably infinitely many copies of $G$.

    \subsection{Fraïssé theory and EPPA}
    We will routinely use Fraïssé-theoretic terminology in the paper. For a detailed account, see e.g.\ \cite[Section 7.1]{Hod}. Most important definitions and facts are the following.

    By a \emph{hereditary class} we will mean a class of finitely generated structures in a fixed signature (in our case, usually, in the language of groups) which is closed under taking finitely generated substructures and isomorphism.

    Given a hereditary class $\cK$, the \emph{joint embedding property} (shortly, JEP) means that any two elements of $\cK$ embed into another element of $\cK$. Similarly, \emph{amalgamation property} (AP) means that for any three structures $A,B,C\in\cK$ and fixed embeddings of $A$ into $B$ and $C$, we can find embeddings of $B$ and $C$ into some $D\in\cK$ such that the two induced embeddings of $A$ into $D$ coincide.

    We say that a hereditary class is \emph{essentially countable} if there are countably many isomorphism types of elements of $\cK$.

    The \emph{age} of a structure $M$, denoted $\Age(M)$, is the class of finitely generated structures which embed into $M$.

    A \emph{finite partial isomorphism} $M\to N$ is either an isomorphism between two finitely generated substructures of $M$ and $N$ or a finite partial function from $M$ to $N$ which preserves the truth of all quantifier-free formulas (the two notions are essentially equivalent). When $M=N$, we call it a \emph{finite partial automorphism}. A partial isomorphism and partial automorphism is defined the same way, only without the cardinality restrictions.

    We say that a structure $M$ is \emph{ultrahomogeneous} if every finite partial automorphism of $M$ can be extended to an automorphism of $M$.

    Now, given a structure $M$, we have the following:
    \begin{itemize}
        \item
        $\Age(M)$ is a hereditary class with JEP,
        \item
        if $M$ is ultrahomogeneous, then $\Age(M)$ has the AP,
        \item
        if $M$ is countable, then $\Age(M)$ is essentially countable,
        \item
        if $\cK$ is a hereditary class which is essentially countable, has the JEP and AP, then there is a unique countable ultrahomogeneous structure whose age is $\cK$, which we dub the \emph{Fraïssé limit of $\cK$}.
    \end{itemize}

    The notion of EPPA (extension property for partial automorphisms, also called Hrushovski's property) will not be used directly in this paper, but we recall one variant here, since it motivates the crucial for us notion of inner EPPA (see Definition~\ref{dfn:inner_eppa}).
    \begin{definition}
        We say that a hereditary class $\cK$ has 1-EPPA if for any $A\in\cK$ and finite partial automorphism $p$ of $A$, there is a $B\in\cK$ extending $A$ and a $\sigma\in\Aut(B)$ which extends $p$.
    \end{definition}

    \subsection{Amalgamated free products and HNN extensions}
    The notions of free products and HNN extensions are very relevant to some of the main motivating examples (although the main technical results do not rely on them as much). We recall some of the most important notions here. For more background, see \cite{LS01}. Some rudimentary Bass-Serre theory will also be useful for some observations, see e.g.\ \cite{Bas93} for background on that.

    Let $A,B,C$ be groups with fixed embeddings $A\to B,C$. Then the amalgamated free product $B*_AC$ is the unique (up to isomorphism) group in which $B,C$ embed in such a way that any homomorphisms $B\to D$, $C\to D$ which agree on (the images of) $A$ extend uniquely to a homomorphism $B*_AC\to D$. We will call the amalgamated free product \emph{finitary} if the base $A$ is finitely generated.
    Note that if we identify $B$ and $C$ with their images in $B*_AC$, then $B\cap C=A$.
    When $A$ is trivial, we write simply $B*C$ and we call this the \emph{free product} of $B$ and $C$.

    If $G$ is a group and $\alpha$ is a partial automorphism $G\to G$, then the HNN-extension $G*_\alpha=\langle G,t\rangle$ of $G$ by $\alpha$ is the universal extension of $G$ by a new generator (called a \emph{stable letter}) $t$ such that for $g\in\dom \alpha$ we have $g^t=\alpha(g)$ (so that $G*_p$ has a unique homomorphism onto any other such group). We call a HNN extension \emph{finitary} if $p$ is a finite partial automorphism (i.e.\ its domain is finitely generated).
    We similarly define HNN extensions given by an arbitrary family of partial automorphisms of $G$ (in which case, there is a separate new stable letter for each partial automorphism).

    Note that in particular, we have for each $G$ the HNN-extensions $G*_{\id G}=G\times\bZ$ and $G*_{\emptyset}=G*\bZ$.

    Both (amalgamated) free products and HNN extensions have well-known descriptions in terms of group presentations, and in each case, the factor groups naturally embed into the product.

    We finish this section by noting an important nontrivial property of HNN-extensions, which will be useful in some examples.
    \begin{fact}[Torsion theorem for HNN extensions]
        \label{fct:torsion_hnn}
        Given any group $G$ and set $P$ of partial automorphisms of $G$, every element of $G*_P$ of finite order is conjugate to an element of $G$.
    \end{fact}
    \begin{proof}
        See \cite[Theorem 2.4 of Chapter IV]{LS01} for the case of $P=\{p\}$. The general case follows by induction.
    \end{proof}

    \subsection{Amalgamation and partial automorphisms of finite groups}
    In the paper \cite{Hal59}, Hall defined the universal countable locally finite group $\Gamma_F$. In the same paper, he proved the following:
    \begin{fact}
        \label{fct:finite_hnn}
        If $A$ is a finite group and $p$ is a finite partial automorphism of $A$, then there is a finite group $B\geq A$ and an element $b\in B$ such that $a^b=p(a)$ for $a\in \dom p$.
    \end{fact}
    \begin{proof}
        This is \cite[Lemma 1]{Hal59}.
    \end{proof}

    Hall's construction, via Fraïssé theory, easily implies that the class of finite groups has the amalgamation property. However, an explicit construction of the amalgam, using permutation groups, was given by B.H.\ Neumann.
    \begin{fact}
        \label{fct:finite_amalg}
        If $A\leq B,C$ are finite groups, then there is a finite group $D$ in which $B,C$ embed in such a way that the intersection of the two embeddings is $A$.
    \end{fact}
    \begin{proof}
        The construction (and the proof of this property) is given in \cite[Section 2]{Neu60}.
    \end{proof}

    \subsection{Model theory}
    Purely model-theoretic terminology will only be used in Section~\ref{sec:mt} and in discussing the model-theoretic properties of specific examples. We will not define most of it in this paper, including type spaces, saturated models, the properties of smallness, stability, pseudofiniteness, quantifier elimination, decidability, independence property, $n$-independence property, or the tree property of the second kind. See e.g.\ \cite{Hod} for general background on model theory, \cite{Che14} for the tree property of the second kind, and \cite{Sh14} for $n$-(in)dependence property.

    \section{Definitions and preliminary observations}
    \subsection{Inner EPPA and ultrahomogeneity}
    In this section, we will define the main notions of this paper.
    \begin{definition}
        \label{dfn:inner_eppa}
        We say that a class $\cK$ of groups has \emph{inner EPPA} if for every $A\in\cK$ and every finite partial automorphism $p$ of $A$, there is a $B\geq A$ in $\cK$ and an element $b\in B$, a \emph{witness of inner EPPA} for $p$, such that for each $a\in\dom p$ we have $a^b=p(a)$.
    \end{definition}

    \begin{definition}
        \label{dfn:inner_uh}
        We say that a group $\Gamma$ is \emph{inner ultrahomogeneous} if for every finite partial automorphism $p$ of $\Gamma$ there is some $g\in \Gamma$, a \emph{witness of inner ultrahomogeneity for $p$} such that for $a\in\dom p$ we have $a^g=p(a)$.
    \end{definition}

    In this paper, unless stated otherwise, the capital letter $\Gamma$ will be used to refer to an inner ultrahomogeneous group (but most of the time, we will repeat this hypothesis when necessary).

    \begin{remark}
        \label{rem:conj_order}
        If $\Gamma$ is inner ultrahomogeneous, then $a,b\in \Gamma$ are conjugate if and only if they have the same order (because then $a\mapsto b$ is clearly a partial automorphism).
    \end{remark}

    \begin{proposition}
        \label{prop:char_inner_uh}
        If $\Gamma$ is any group, then $\Gamma$ is inner ultrahomogeneous if and only if it is ultrahomogeneous and $\Age(\Gamma)$ has inner EPPA. In particular, the Fraïssé limit of a Fraïssé class of groups with inner EPPA is inner ultrahomogeneous.
    \end{proposition}
    \begin{proof}
        The fact that inner ultrahomogeneity of $\Gamma$ implies ultrahomogeneity and inner EPPA for $\Age(\Gamma)$ is trivial. Suppose now that $\Gamma$ is ultrahomogeneous and $\Age(\Gamma)$ has inner EPPA. Fix a finite partial automorphism of $\Gamma$. Let $A\leq \Gamma$ be generated by the domain and range of $p$, so that $p$ is a partial automorphism of $A$. Let $B\geq A$ in $\Age(\Gamma)$ contain a witness $b$ of inner EPPA for $p$. Then by inner ultrahomogeneity, we may assume that $A\leq B\leq \Gamma$, and then the same $b$ works as a witness of inner ultrahomogeneity.
    \end{proof}

    \begin{corollary}
        If $\Gamma$ is ultrahomogeneous and $\Gamma_0\leq \Gamma$ is inner ultrahomogeneous and $\Age(\Gamma)=\Age(\Gamma_0)$, then $\Gamma$ is inner ultrahomogeneous.
    \end{corollary}
    \begin{proof}
        Immediate by the preceding proposition.
    \end{proof}

    HNN-extensions show that the class of all finitely generated groups has inner EPPA, and it is not hard to see that so do the classes of finitely presentable and finitely generated recursively presentable groups. Fact~\ref{fct:finite_hnn} shows that so does the class of finite groups. This leads to the groups $\Gamma_R$, $\Gamma_F$ mentioned in the introduction. There are also three finite inner ultrahomogeneous groups, namely the first three symmetric groups. See Section~\ref{sec:examples} for more details and more examples.

    \subsection{Implications between properties}
    The following proposition establishes some relations between some properties of hereditary classes of groups. This will be useful mainly to simplify the description of examples in Section~\ref{sec:examples}.

    \begin{proposition}
        \label{prop:implications_between_props}
        Let $\cK$ be a hereditary class of groups with the joint embedding property (e.g.\ the age of a fixed group). Then:
        \begin{enumerate}
            \item
            if $\cK$ is closed under $*\bZ$ if and only if it contains a nontrivial group and is closed under (binary) free products,
            \item
            if $\cK$ is closed under $*\bZ$ and has inner EPPA, then it is closed under $\times\bZ$;
            \item
            $\cK$ is closed under finitary HNN-extensions, then it has inner EPPA and is closed under finitary amalgamated free products (so it has the amalgamation property).
        \end{enumerate}
    \end{proposition}
    \begin{proof}
        For (1), in one direction, note that if $A$ is a nontrivial group, then $A*A$ is not torsion, so it contains a copy of $\bZ$. In the other direction, suppose $A,B\in \cK$ and $\cK$ is closed under $*\bZ$. Let $C$ be a supergroup of $A$ and $B$ in $\cK$ (it exists, because $\cK$ has the JEP). Then $A*B$ embeds into $C*C$, which embeds into $C*\bZ$ as $\langle C,C^t\rangle$, where $t$ is the generator of $\bZ$.

        For (2), if $\cK$ has inner EPPA and $A*\bZ\in\cK$, then this group has a partial automorphism which fixes $A$ and squares the generator of $\bZ$. Given $B\geq A*\bZ$ and element $g\in B$ witnessing inner ultrahomogeneity for this partial automorphism, it is easy to see $\langle A,g\rangle\cong A\times\bZ$.

        Now, for (3), suppose $\cK$ is closed under finitary HNN-extensions. Then clearly $\cK$ has inner EPPA (immediately by definition) and is closed under $*\bZ$ (by taking a trivial partial automorphism), so by (1), it is closed under free products. To see that it is closed under finitary amalgamated free products, fix $B,C\in\cK$ and their isomorphic f.g.\ subgroups $A_B,A_C$. Let $G$ be a graph of groups with vertices $B$ and $C$ and two edges, one of which has trivial edge group connected to $B$ and $C$, and the other embedding into $B$ and $C$ as $A_B$ and $A_C$, respectively. Then by elementary Bass-Serre theory we have $(B*_{A_B=A_C}C)*\bZ\cong \pi_1(G)\cong (B*C)*_{A_C^t=A_B}$. Since the right hand side is in $\cK$, the conclusion follows.
    \end{proof}

    \begin{corollary}
        \label{cor:hnn_uh}
        If $\Gamma$ is an ultrahomogeneous group and $\Age(\Gamma)$ is closed under finitary HNN-extensions, then it is inner ultrahomogeneous.
    \end{corollary}
    \begin{proof}
        Immediate by Proposition~\ref{prop:implications_between_props} and Proposition~\ref{prop:char_inner_uh}.
    \end{proof}

    \begin{corollary}
        \label{cor:hnn_Fraïssé}
        If $\cK$ is an essentially countable hereditary class of groups which is closed under finitary HNN-extensions and has JEP (e.g.\ it is closed under direct products or under free products), then $\cK$ is a Fraïssé class with inner EPPA.
    \end{corollary}
    \begin{proof}
        Immediate by Proposition~\ref{prop:implications_between_props}.
    \end{proof}

    \subsection{Abelian groups of automorphisms}
    In this section, we will establish some lemmas which will allow us to find abelian subgroups of inner ultrahomogeneous groups, which will be very useful to obtain information about their algebraic structure.

    For a group $A$, we will write $\Emb(A)$ for the monoid of self-embeddings of $A$ (i.e.\ injective endomorphisms).

    \begin{lemma}
        \label{lem:commuting_conjugation_group}
        Suppose $A,B\leq \Gamma$ and $B\leq N(A)$. Suppose furthermore that $D=BA$ ($=\langle A,B\rangle$, since $B$ normalises $A$) is finitely generated and $\sigma\in \Emb(A)$ commutes with conjugation by elements of $B$ and fixes all elements of $A\cap B$.

        Then there is a $g\in C(B)$ such that for all $a\in A$ we have $a^g=\sigma(a)$.
    \end{lemma}
    \begin{proof}
        We will show that $\bar\sigma(ba)=b\sigma(a)$ yields a well-defined self-embedding of $D$. Then the conclusion will follow immediately by inner ultrahomogeneity.

        Suppose $b_1a_1=b_2a_2$. Then $a_2a_1^{-1}=b_2^{-1}b_1\in A\cap B$, so $\sigma(a_2)\sigma(a_1)^{-1}=\sigma(a_2a_1^{-1})=a_2a_1^{-1}=b_2^{-1}b_1$, whence $b_1\sigma(a_1)=b_2\sigma(a_2)$, so $\bar\sigma$ is well-defined.

        To see that $\bar\sigma$ is a homomorphism, note that $b_1a_1b_2a_2=b_1b_2a_1^{b_2}a_2$, so $b_1\sigma(a_1)b_2\sigma(a_2)=b_1b_2\sigma(a_1)^{b_2}\sigma(a_2)=b_1b_2\sigma(a_1^{b_2})\sigma(a_2)=b_1b_2\sigma(a_1^{b_2}a_2)$.

        Finally, to see that $\bar\sigma$ is injective, observe that if $b\sigma(a)=1$, then $\sigma(a)=b^{-1}\in A\cap B$, so $\sigma(a)=\sigma(\sigma(a))$, whence $\sigma(a)=a$ (by injectivity of $\sigma$), so $a=b^{-1}$ and $ba=1$.
    \end{proof}

    \begin{lemma}
        \label{lem:commuting_automorphisms_with_fixed_points}
        Suppose $A\leq\Gamma$ is finitely generated, $\Sigma\subseteq\Emb(A)$ consists of commuting elements and there is a $\sigma_0\in\Sigma$ such that every $\sigma\in \Sigma$ fixes all the fixed points of $\sigma_0$ (and possibly more).

        Then we can assign to each $\sigma\in\Sigma$ some $g_\sigma$ in $\Gamma$ in such a way that $\langle g_\sigma\mid \sigma\in \Sigma\rangle$ is abelian and for $a\in A$ and $\sigma\in\Sigma$ we have $a^{g_\sigma}=\sigma(a)$.
    \end{lemma}
    \begin{proof}
        First, by inner ultrahomogeneity, we can find a $g_0\in \Gamma$ such that for $a\in A$ we have $a^{g_0}=\sigma_0(a)$. We can also assume without loss of generality that $\Sigma$ is a monoid (i.e.\ $\id_A\in\Sigma$ and $\Sigma$ is closed under composition).

        Now, extend $(\sigma_0)$ to $(\sigma_i)_{i\in\bN}$, a (possibly non-injective) enumeration of $\Sigma$ (this exists, because $\Sigma$ is countable, because $A$ is finitely generated).

        We will recursively build a sequence $(g_i)_{i\in\bN}$ of commuting elements of $\Gamma$ such that for $a\in A$ we have $a^{g_i}=\sigma_i(a)$, which will complete the proof.

        Suppose we have commuting $g_0, g_1,\ldots,g_k\in\Gamma$. Let $B_k\coloneqq \langle g_0,\ldots,g_k\rangle$ and $A_k\coloneqq \langle A^{B_k}\rangle$.
        We want to extend $\sigma_{k+1}$ to a self-embedding of $A_k$ which commutes with conjugation by elements of $B_k$ and fixes all elements of $A_k\cap B_k$. Since $\langle A_k,B_k\rangle=\langle A,g_0,\ldots,g_k\rangle$ is finitely generated, we can then apply Lemma~\ref{lem:commuting_conjugation_group} to obtain $g_{k+1}$ as desired.

        Using the fact that $B_k$ is a finitely generated abelian group, it is not hard to see that for any $b_1,b_2\in B_k$ there is $b\in B_k$ (``$-\min(b_1,b_2)$'') and $\tau_1,\tau_2\in \Sigma$ such that for all $a\in A$ we have $a^{b_i-b}=\tau_i(a)$ for $i=1,2$.

        It follows that for any $a_1,a_2\in A$ we have $a_1^{b_1}a_2^{b_2}=(\tau_1(a_1)\tau_2(a_2))^b\in A^{B_k}$, so $A_k=A^{B_k}$. Now, we want to show that $\bar\sigma_{k+1}(a^b)=\sigma_{k+1}(a)^b$ defines the desired embedding.

        First, to see that it is well-defined, suppose $a_1^{b_1}=a_2^{b_2}$ and let $b,\tau_1,\tau_2$ be as before, so $a_i^{b_i}=\tau_i(a_i)^b$ for $i=1,2$. By commutativity of $\Sigma$, we have for $i=1,2$ that $\sigma_{k+1}\circ \tau_i(a_i)=\tau_i(\sigma_{k+1}(a_i))=\sigma_{k+1}(a_i)^{b_i-b}$, so $\sigma_{k+1}(a_i)^{b_i}=\sigma_{k+1}(\tau_i(a_i))^b$. Since $\tau_1(a_1)=(a_1^{b_1})^{-b}=(a_2^{b_2})^{-b}=\tau_2(a_2)$, we have that $\sigma_{k+1}(a_1)^{b_1}=\sigma_{k+1}(a_2)^{b_2}$.

        Now, to see that $\bar\sigma_{k+1}$ is a homomorphism, choose any $a_1,a_2\in A$ and $b_1,b_2\in B_k$, along with $\tau_1,\tau_2$ as before. Then
        \begin{multline*}
            \bar\sigma_{k+1}(a_1^{b_1}a_2^{b_2})=\bar\sigma_{k+1}((\tau_1(a_1)\tau_2(a_2))^b)=\\
            =\sigma_{k+1}(\tau_1(a_1)\tau_2(a_2))^b=
            \sigma_{k+1}(\tau_1(a_1))^b\sigma_{k+1}(\tau_2(a_2))^b=\\
            =\bar\sigma_{k+1}(\tau_1(a_1)^b)\bar\sigma_{k+1}(\tau_2(a_2)^b)
            =\bar\sigma_{k+1}(a_1^{b_1})\bar\sigma_{k+1}(a_2^{b_2}).
        \end{multline*}
        The fact that the kernel of $\bar\sigma_{k+1}$ is trivial is immediate by definition of $\bar\sigma_{k+1}$ and the injectivity of $\sigma_{k+1}$. The fact that $\bar\sigma_{k+1}$ commutes with conjugation by elements of $B_k$ follows immediately from the fact that it is well-defined.

        Finally, if $a^b\in B_k\cap A^{B_k}$, then $a^b$ commutes with $B_k$, whence $a^b=a$ and $a$ commutes with $g_0$, so it is fixed by $\sigma_0$. By hypothesis, it is fixed by every element of $\Sigma$, which includes $\sigma_{k+1}$. Thus $\bar\sigma_{k+1}(a^b)=\sigma_{k+1}(a)=a$.
    \end{proof}
    \begin{corollary}
        \label{cor:abelian_automorphism_group_semidirect}
        Suppose $\Gamma$ is torsion inner ultrahomogeneous, $A\leq\Gamma$ is finite, while $B\leq\Aut(A)$ is abelian, of order coprime with the order of $A$. Then there is a finite abelian $C\leq N(A)$ such that $B$ is the image of $C$ via the conjugation map $N(A)\to \Aut(A)$ and the orders of $A$ and $C$ are coprime. (In particular, $\langle A,C\rangle$ is the (internal) semidirect product of $A$ and $C$ and we have the natural epimorphism $\langle A,C\rangle\to A\rtimes B$.)
    \end{corollary}
    \begin{proof}
        Enumerate $B$ as $\sigma_1,\ldots,\sigma_n$. We will recursively find commuting $g_1,\ldots,g_n\in N(A)$ such that for $a\in A$, $j=1,\ldots, n$ we have $a^{g_j}=\sigma_j(a)$ and for each $k$, the order of $C_k\coloneqq\langle g_1,\ldots, g_k\rangle$ is coprime with the order of $A$.

        Suppose $k<n$ and we already have $g_1,\ldots,g_k$ satisfying the above. It follows that $C_k\cap A=\{1\}$ (since every non-identity element of $C_k$ has order not dividing the order of $A$). Thus, the hypotheses of Lemma~\ref{lem:commuting_conjugation_group} hold (with $B=C_k$, $\sigma=\sigma_{k+1}$), so there is a $g\in C(C_k)$ such that $a^g=\sigma_{k+1}(a)$.

        Now, let $n$ be such that $n\lvert A\rvert\equiv 1$ modulo the order of $\sigma_{k+1}$ (it exists by hypothesis). Since $\Gamma$ is torsion, the order of $g$ is finite, and it follows that for a sufficiently large $K$, the order of $g_{k+1}\coloneqq g^{(n\lvert A\rvert)^K}$ is coprime with the order of $A$, and for $a\in A$ we have $a^{g^{(n\lvert A\rvert)^K}}=\sigma_{k+1}^{(n\lvert A\rvert)^K}(a)=\sigma_{k+1}^{1^K}(a)=\sigma_{k+1}(a)$.

        Finally, since $g_{k+1}\in C(C_k)$, the order of $C_{k+1}\coloneqq \langle C_k,g_{k+1}\rangle$ divides the product of the order of $C_k$ and the order of $g_{k+1}$, so it is coprime with the order of $A$, which completes the induction step.

        $C\coloneqq C_n$ clearly works. The parenthetical remark follows from the fact that $A\cap C_n=\{1\}$ and $C_n\leq N(A)$.
    \end{proof}

    The following fact will be especially useful for extracting finite abelian subgroups of torsion inner ultrahomogeneous groups.
    \begin{fact}
        \label{fct:finite_abelian_quotient}
        If $B$ is a finite abelian group and $A\leq B$, then there is a $B_0\leq B$ such that $B_0\cong B/A$.
    \end{fact}
    \begin{proof}
        This is \cite[Exercise 43 of Chapter I]{Lang}, follows from the structure theorem for finite abelian groups.
    \end{proof}

    \begin{corollary}
        \label{cor:abelian_automorphism_group_direct}
        If $\Gamma$ is torsion inner ultrahomogeneous, $A\leq \Gamma$ is finite, $B\leq \Aut(A)$ is abelian of order coprime with the order of $A$ and $A_0\leq A$ is abelian and fixed pointwise by $B$. Then $\Gamma$ has a subgroup isomorphic to $A_0\times B$.
    \end{corollary}
    \begin{proof}
        Let $C$ be as in Corollary~\ref{cor:abelian_automorphism_group_semidirect}. Then we have a natural epimorphism $\langle A_0,C\rangle\to A_0\rtimes B=A_0\times B$ (because $B$ acts trivially on $A_0$). Furthermore, $C$ clearly centralises $A_0$, so $\langle A_0,C\rangle$ is abelian. Since $\langle A_0,C\rangle$ is finite abelian (because it is finitely generated abelian and torsion), by Fact~\ref{fct:finite_abelian_quotient}, it has a subgroup isomorphic to $A_0\times B$.
    \end{proof}

    \begin{corollary}
        \label{cor:abelian_automorphism_group}
        If $\Gamma$ is torsion inner ultrahomogeneous, $A\leq \Gamma$ is finite and $B\leq \Aut(A)$ is finite abelian of order coprime with the order of $A$, then $\Gamma$ has a subgroup isomorphic to $B$.
    \end{corollary}
    \begin{proof}
        Immediate by Corollary~\ref{cor:abelian_automorphism_group_direct} with trivial $A_0$.
    \end{proof}

    \begin{remark}
        In some of the above corollaries, we can weaken the hypothesis to only assume that $N(A)$ is torsion, and not necessarily $\Gamma$, and $A$ is possibly infinite, but $B$ has no element of the same order as a nonidentity element of $A$. Then the $C$ in the conclusion will have the same property.
    \end{remark}

    \section{Group-theoretic properties}

    \subsection{Groups of finite exponent}
    An inner ultrahomogeneous group can have exponent $1, 2$ or $6$, as witnessed by the finite ultrahomogeneous groups.

    However, it is not clear whether an infinite ultrahomogeneous group can have finite exponent. In this section, we provide (in Theorem~\ref{thm:char_of_finite_exp}) several characterisations of the inner ultrahomogeneous groups of finite exponent in terms of their abelian subgroups.

    \begin{lemma}
        \label{lem:2_explosion}
        Suppose $\Gamma$ is torsion inner ultrahomogeneous and it has a subgroup isomorphic to $(\bZ/2\bZ)^6$. Then for each $n$, it has a subgroup isomorphic to $(\bZ/2\bZ)^n$.
    \end{lemma}
    \begin{proof}
        Fix any $n$. Suppose $G=(\bZ/2\bZ)^n$ embeds into $\Gamma$.

        Put $m\coloneqq \lfloor n/2\rfloor$. Write elements of $G$ as $(v_1,w,v_2)$, where $v_1,v_2\in (\bZ/2\bZ)^m$ and $w$ is either an empty tuple or in $\bZ/2\bZ$ (depending on whether $n$ is even or odd).

        For each $f\in \End((\bZ/2\bZ)^m)$, let $\sigma_f$ be the endomorphism of $G$ given by $\sigma_f(v_1,w,v_2)=(v_1+f(v_2),w,v_2)$. Observe that $\sigma_f\circ\sigma_g=\sigma_{f+g}$.
        Thus $\sigma_{-}$ defines an isomorphism between the additive group of $\End((\bZ/2\bZ)^m)$ (which is elementary abelian of order $2^m=2^{\lfloor n/2\rfloor}$) and a subgroup of $\Aut(G)$.
        Furthermore, they all fix elements of the form $(v_1,w,0)$, and it is not hard to see that $\sigma_0=\sigma_{\id}$ fixes no other elements.

        Thus, by Lemma~\ref{lem:commuting_automorphisms_with_fixed_points}, there is an abelian (and hence finite abelian, since $\Gamma$ is torsion) $B\leq\Gamma$ which admits an epimorphism onto $(\bZ/2\bZ)^{\lfloor n/2\rfloor^2}$, so by Fact~\ref{fct:finite_abelian_quotient}, the group $(\bZ/2\bZ)^{\lfloor n/2\rfloor^2}$ embeds into $\Gamma$.

        Since $\lfloor n/2\rfloor^2>n$ for $n\geq 6$, the conclusion follows.
    \end{proof}

    \begin{remark}
        In fact, $\Aut((\bZ/2\bZ)^5)$ has a subgroup $\Sigma$ isomorphic to $(\bZ/2\bZ)^6$, but there is no $\sigma_0\in \Sigma$ which fixes all the common fixed points of $\Sigma$, so Lemma~\ref{lem:commuting_automorphisms_with_fixed_points} does not apply.
    \end{remark}

    \begin{lemma}
        \label{lem:cyclic_2_group}
        Suppose $\Gamma$ is torsion inner ultrahomogeneous. Then for $k\geq 2$, $m\geq 0$:
        \begin{itemize}
            \item
            if the group $(\bZ/2^k\bZ)\times (\bZ/2\bZ)^m$ embeds into $\Gamma$, then so does $(\bZ/2^{k-2}\bZ)\times (\bZ/2\bZ)^{m+1}$,
            \item
            if the group $\bZ/2^k\bZ$ embeds into $\Gamma$, then so does $(\bZ/2\bZ)^{\lceil k/2\rceil}$.
        \end{itemize}
    \end{lemma}
    \begin{proof}
        The second dot follows from the first by straightforward induction. Consider first the first dot, in the case of $k\geq 3$ (the case of $k=2$ is trivial). Consider the following automorphisms of $G=(\bZ/2^k\bZ)\times (\bZ/2\bZ)^m$:
        \begin{itemize}
            \item
            $\sigma_1(a,j_1,\ldots,j_m)=(3a,j_1,\ldots,j_m)$,
            \item
            $\sigma_2(a,j_1,\ldots,j_m)=(-a,j_1,\ldots,j_m)$,
            \item
            $\tau_i(a,j_1,\ldots,j_m)=(a,j_1,\ldots,j_{i-1},j_i+a,j_{i+1},\ldots,j_m)$, where $i=1,\ldots,m$.
        \end{itemize}
        (Here, we operate mod $2^{k}$ on the first coordinate and mod $2$ on the others.)

        It is easy to see that $\sigma_1,\sigma_2$ and all $\tau_i$ commute, it is well-known that the order of $\sigma_1$ (= the order of $3\in (\bZ/2^k\bZ)^\times$) is $2^{k-2}$ and $\sigma_2\notin \langle\sigma_1\rangle$; it is easy to see that the orders of all the others are all equal to $2$. Using this, it is not hard to see that $\sigma_1,\sigma_2,\tau_1,\ldots,\tau_m$ naturally generate a group isomorphic to $(\bZ/2^{k-2}\bZ)\times(\bZ/2\bZ)^{m+1}$.

        Then $(2^{k-1}\bZ/2^k\bZ)\times (\bZ/2\bZ)^m$ is the set of fixed points of $\sigma_2$, and it is easy to see that its elements are fixed by $\sigma_1$ and each $\tau_i$. The conclusion follows from Lemma~\ref{lem:commuting_automorphisms_with_fixed_points} and Fact~\ref{fct:finite_abelian_quotient}.
    \end{proof}

    \begin{corollary}
        \label{cor:2_group_to_vector}
        If $\Gamma$ is torsion inner ultrahomogeneous and $K$ is such that $\Gamma$ has no subgroup isomorphic to $(\bZ/2\bZ)^{K+1}$, then it $\Gamma$ has no abelian $2$-subgroup of order greater than $2^{2K^2}$.
    \end{corollary}
    \begin{proof}
        Let $A\leq\Gamma$ be an abelian $2$-group. We may assume that it is finite (since every abelian $2$-group is locally finite).

        By fundamental theorem on finite abelian groups, $A$ is the direct sum of some $m$ nontrivial cyclic $2$-groups. Clearly, there can be no more than $K$ of them. Furthermore, by Lemma~\ref{lem:cyclic_2_group}, each of them has order at most $2^{2K}$. The conclusion follows.
    \end{proof}

    \begin{lemma}
        \label{lem:odd_prime_divisor}
        Suppose $\Gamma$ is torsion and inner ultrahomogeneous. Let $A\leq \Gamma$ be a finite abelian subgroup of order divisible by an odd prime $p$. Then $\Gamma$ has a finite abelian subgroup $B$ of order divisible by $\frac{p-1}{p}\lvert A\rvert$. (In particular, it has an abelian subgroup of order exactly $\frac{p-1}{p}\lvert A\rvert$.)
    \end{lemma}
    \begin{proof}
        By the fundamental theorem on finitely generated abelian groups, $A\cong (\bZ/p^k\bZ)\times A'$ for some $A'$; for simplicity, suppose $A=(\bZ/p^k\bZ)\times A'$. Note that the automorphism group of $\bZ/p^k\bZ$ is cyclic of order $p^{k-1}\cdot(p-1)$. It follows that there is some $\sigma\in\Aut(A)$ of order $p^{k-1}\cdot (p-1)$ which fixes $A'$. By inner ultrahomogeneity, there is a $g\in \Gamma$ such that $a^g=\sigma(a)$ for $a\in A$.

        We claim that $B=\langle g,A'\rangle$ works. Indeed, it is clearly abelian, and the conjugation map $B\to \Aut(A)$ contains $A'$ in its kernel and $\sigma$ in its image. Hence the order of $B$ is a multiple of the order of $A'$ (=$\frac{1}{p^k}\lvert A\rvert$) times the order of $\sigma$ (=$p^{k-1}\cdot (p-1)$). The conclusion follows. (The parenthetical remark easily follows from the observation that if $A$ is finite abelian and $q$ is a prime dividing $\lvert A\rvert$, then $A$ has a subgroup of index $q$.)
    \end{proof}

    \begin{corollary}
        \label{cor:finite_group_to_2_group}
        If $\Gamma$ is torsion inner ultrahomogeneous and it has no abelian $2$-subgroup of order $2^{K+1}$, then it has no abelian subgroup of order greater than $2^{2K}$.
    \end{corollary}
    \begin{proof}
        Note that if $2^k$ divides the order of a finite abelian group $A$, then $A$ has a subgroup of order $2^k$.

        Fix any abelian $A\leq\Gamma$; we will show that $\lvert A\rvert\leq 2^{2K}$. Since a torsion abelian group is locally finite, we may assume without loss of generality that $A$ is finite. We recursively define a finite sequence of abelian groups $A_j\leq \Gamma$ and integers $l_j,n_j$ such $n_j$ is odd and the order of $A_j$ is $2^{l_j}n_j$, starting with $A_0\coloneqq A$.

        Suppose we already have $A_0,\ldots, A_j$.
        If $n_j=1$, then we terminate. Otherwise, let $p_{j+1}$ be the largest prime divisor of $n_j$ and write $p_{j+1}=2^{k_{j+1}}m_{j+1}+1$, where $m_{j+1}$ is odd. Then by Lemma~\ref{lem:odd_prime_divisor}, there is a $A_{j+1}\leq \Gamma$ of order $\frac{2^{l_j}n_j(p_{j+1}-1)}{p_{j+1}}=2^{l_j+k_{j+1}}m_{j+1}\frac{n_j}{p_{j+1}}$. Put $n_{j+1}=m_{j+1}\frac{n_j}{p_{j+1}}$ and $l_{j+1}=l_j+k_{j+1}$, so that the order of $A_{j+1}$ is $2^{l_{j+1}}n_{j+1}$, as prescribed.

        Note that the sequence $l_j$ is strictly increasing and bounded by $K$. It follows that we terminate after at most $K$ steps, yielding some $A_{j_0}$ with $j_0\leq K$ which is an abelian $2$-group of order $2^{l_{j_0}}$, where $l_{j_0}=l_0+k_1+\ldots+k_{j_0}\leq K$. On the other hand, for $j=1,\ldots,j_0$, we have $n_{j}=m_j\frac{n_{j-1}}{p_j}=\frac{n_{j-1}(p_j-1)}{2^{k_j}p_j}\geq n_{j-1}\cdot 2^{-1-k_j}$.
        It follows that $n_{j}\geq n_0\cdot 2^{-k_1-k_1-\ldots-k_j-j}$, so in particular, $1=n_{j_0}\geq n_0\cdot 2^{-k_1-k_2-\ldots-k_{j_0}-j_0}$, so
        $n_0\leq 2^{k_1+k_2+\ldots+k_{j_0}+j_0}\leq 2^{2K-l_0}$, whence $\lvert A_0\rvert=n_0\cdot 2^{l_0}\leq 2^{2K}$.
    \end{proof}

    \begin{theorem}
        \label{thm:char_of_finite_exp}
        If $\Gamma$ is torsion inner ultrahomogeneous, then the following are equivalent:
        \begin{enumerate}
            \item
            all abelian subgroups of $\Gamma$ have order smaller than $2^{100}$,
            \item
            $\Gamma$ is of finite exponent (at most $(2^{100})!$),
            \item
            there is some finite abelian group which does not embed into $\Gamma$,
            \item
            $(\bZ/2\bZ)^6$ does not embed into $\Gamma$.
        \end{enumerate}
    \end{theorem}
    \begin{proof}
        It is clear that 1$\Rightarrow$2$\Rightarrow$3.

        For 3$\Rightarrow$4 we argue by contraposition. Suppose (4) fails. Fix a finite abelian group $A$; we claim that it embeds into $\Gamma$. Let $n=\lvert A\rvert$. Consider first the case when $n$ is odd, and note that by Lemma~\ref{lem:2_explosion}, $\Gamma$ has a subgroup $B\cong (\bZ/2\bZ)^n$. Note that $A$ acts faithfully on $B$,
        and their orders are coprime. By Corollary~\ref{cor:abelian_automorphism_group}, $\Gamma$ has a subgroup isomorphic to $A$. Otherwise, if $n$ is even, let $p$ be a prime not dividing $\lvert A\rvert$. Then $p^n$ is odd, so by the odd $n$ case, $\Gamma$ has a subgroup isomorphic to $(\bZ/p\bZ)^n$ on which $A$ acts faithfully, and similarly, Corollary~\ref{cor:abelian_automorphism_group} implies that $A$ embeds into $\Gamma$.

        It remains to show that 4$\Rightarrow$1. But its contrapositive is an easy consequence of  Corollary~\ref{cor:finite_group_to_2_group} and Corollary~\ref{cor:2_group_to_vector}.
    \end{proof}

    \begin{remark}
        The explicit estimates in Theorem~\ref{thm:char_of_finite_exp} are very rough and their proofs can easily be refined to obtain much lower ones. It is also not very difficult to see that instead of $(\bZ/2\bZ)^6$ in (4) we can put $(\bZ/p\bZ)^6$ for any prime $p$, or indeed $G^6$ for any nontrivial finite $G$. Optimistically, one can conjecture that the only inner ultrahomogeneous groups of finite exponent are the three finite ones, which would be improve the bounds a great deal.
    \end{remark}

    \begin{remark}
        \label{rem:centr_orbit}
        Inner ultrahomogeneity easily implies that inner automorphisms are dense in $\Aut(\Gamma)$, which implies in particular that if $A\leq \Gamma$ is finitely generated, then orbits of $\Aut(\Gamma/A)$ (the stabiliser of $A$) in $\Gamma$ and its finite powers are exactly the orbits of $C(A)$ acting by conjugation.
    \end{remark}

    \begin{proposition}
        \label{prop:fin_exp_almost_fg}
        If $\Gamma$ is inner ultrahomogeneous of finite exponent, then $\Aut(\Gamma)$ is discrete, i.e.\ there is a finitely generated $A\leq \Gamma$ such that $\Aut(\Gamma/A)$ is trivial and every automorphism of $\Gamma$ is inner (in particular, $\Aut(\Gamma)\cong \Gamma/Z(\Gamma)$).
    \end{proposition}
    \begin{proof}
        Let $\Gamma$ be inner ultrahomogeneous of finite exponent. By Theorem~\ref{thm:char_of_finite_exp}, there is an upper bound on the order of abelian subgroups of $\Gamma$. So there is a finite maximal abelian subgroup $A_0\leq \Gamma$, thus satisfying  $A_0=C(A_0)$. Thus, for each element of $A_0\setminus Z(\Gamma)$, we can find an element of $\Gamma$ which does not commute with it. If we take for $A$ the group generated by $A_0$ and all these witnesses, then we will have $C(A)=Z(\Gamma)$. The conclusion follows by Remark~\ref{rem:centr_orbit}.
    \end{proof}

    \begin{remark}
        \label{rem:fin_exp_trivial_centre}
        Actually, as we will see later in Corollary~\ref{cor:torsion_trivial_centre}, a torsion inner ultrahomogeneous group with more than $2$ elements has trivial centre, and in this case, the conclusion of Proposition~\ref{prop:fin_exp_almost_fg} says that $\Aut(\Gamma)\cong \Gamma$.
    \end{remark}

    \subsection{Torsion-free subgroups}
    Hall's universal group shows that an inner ultrahomogeneous group can be infinite and remain torsion. In this section, we will see that as soon as $\Gamma$ has an element of infinite orders, it has all countable abelian torsion-free groups and all countable non-abelian free groups as subgroups.

    \begin{proposition}
        \label{prop:non-torsion_gives_Zn}
        Suppose $\Gamma$ is inner ultrahomogeneous and it is not torsion. Then it has a subgroup isomorphic to $\bZ^{\oplus\bN}$
    \end{proposition}
    \begin{proof}
        Let $g\in \Gamma$ be of infinite order. Fix any $n$, and let $(p_n)_n$ be a sequence of distinct primes.

        Then for each $j$, we have a $\sigma_j\in\Emb(\langle g\rangle)$ given by $\sigma_j(g)=g^{p_j}$. Then $\sigma_1,\ldots,\sigma_n$ commute and they have no fixed points other than the identity. Thus, Lemma~\ref{lem:commuting_automorphisms_with_fixed_points} yields a set $\{g_n\mid n\in\bN\}$ of commuting elements such that $g^{g_n}=g^{p_n}$ for each $n$. It is easy to see that this implies that they freely generate an abelian subgroup of $\Gamma$.
    \end{proof}

    \begin{corollary}
        \label{cor:all_non_torsion}
        If $\Gamma$ is inner ultrahomogeneous and not torsion, then it has a subgroup isomorphic to $\bQ^{\oplus \bN}$ (a countably infinite direct sum of copies of $\bQ$). In particular, it embeds every countable torsion-free abelian group.
    \end{corollary}
    \begin{proof}
        By the preceding proposition, for each $n$, $\Gamma$ has a subgroup $G_n$ isomorphic to $\bZ^n$.

        Now, for each $n$, we have an embedding of $G_n$ into $n!G_{n+1}$, which extends to an automorphism of $\Gamma$ (by ultrahomogeneity). By reversing it, we obtain a $G_{n}'\geq G_n$ such that $G_n'\cong \bZ^{n+1}$ and $G_n\subseteq n! G_n'$. Going all the way back to $G_0$ we obtain a direct system of subgroups of $\Gamma$ whose union is isomorphic to $\bigcup_{n\in\bN} \frac{1}{n!}\bZ^n=\bQ^{\oplus \bN}$.

        The ``in particular'' follows from the fact that every countable abelian torsion-free group embeds into $\bQ^{\oplus\bN}$.
    \end{proof}

    \begin{corollary}
        \label{cor:torsion_characterisation}
        Let $\Gamma$ be an inner ultrahomogeneous group. Then the following are equivalent:
        \begin{itemize}
            \item
            $\Gamma$ is not torsion,
            \item
            every countable torsion-free abelian group embeds into $\Gamma$,
            \item
            every countable (non-abelian) free group embeds into $\Gamma$.
        \end{itemize}
    \end{corollary}
    \begin{proof}
        The equivalence of the first two bullets is the preceding corollary. The third bullet clearly implies the first. For the converse, it is enough to consider the free group of rank $2$ (since all countable free groups embed into it).

        By the second bullet, $\Gamma$ has a subgroup $G$ isomorphic to $\bZ^2$. It is well-known that $\Aut(\bZ^2)\cong\GL_2(\bZ)$ has a free subgroup of rank $2$. Let $\sigma_1,\sigma_2\in \Aut(G)$ freely generate a subgroup. By inner ultrahomogeneity, there are $g_1,g_2\in \Gamma$ which act on $G$ as $\sigma_1$ and $\sigma_2$, and it is easy to check that they generate a free subgroup of $\Gamma$.
    \end{proof}

    \begin{corollary}
        \label{cor:trichotomy}
        If $\Gamma$ is inner ultrahomogeneous, then at least one of the following holds:
        \begin{itemize}
            \item
            $\Gamma$ is of finite exponent (and all of its abelian subgroups have order smaller than $2^{100}$),
            \item
            $\Gamma$ has a subgroup isomorphic to $(\bQ/\bZ)^{\oplus \bN}$ (and every countable torsion abelian group),
            \item
            $\Gamma$ has a subgroup isomorphic to $\bQ^{\oplus\bN}$ (and every countable torsion-free abelian group, as well as every countable free group).
        \end{itemize}
    \end{corollary}
    \begin{proof}
        If $\Gamma$ is not torsion, then the conclusion is a part of Corollary~\ref{cor:torsion_characterisation}. Otherwise, it is torsion, so by Theorem~\ref{thm:char_of_finite_exp} either the first bullet holds or else every finite abelian group embeds into $\Gamma$, in which case we can use ultrahomogeneity to conclude that the second bullet holds.
    \end{proof}

    \subsection{Centre}
    \begin{proposition}
        \label{prop:center_at_most_2}
        If $\Gamma$ is inner (1-)ultrahomogeneous, then it its centre is trivial or it is cyclic of order $2$, and in the latter case, its nonidentity element is the unique element of order $2$.
    \end{proposition}
    \begin{proof}
        By Remark~\ref{rem:conj_order}, if $g$ is central, then no other element of $\Gamma$ has the same order. This is only possible if the order of $g$ is $2$ or $1$ (otherwise, $g^{-1}\neq g$ has the same order).
    \end{proof}

    \begin{corollary}
        \label{cor:torsion_trivial_centre}
        If $\Gamma$ is torsion with more than $2$ elements, then $Z(\Gamma)$ is trivial.
    \end{corollary}
    \begin{proof}
        Suppose towards contradiction $\Gamma$ is torsion and has a nontrivial centre. Then it has exactly one element of order $2$, so by Lemma~\ref{lem:cyclic_2_group}, it has no elements of order $2^3=8$.

        Note that if $g\in\Gamma$ is not of order $2$ or $1$, then there is some $h\in \Gamma$ such that $g^h=g^{-1}\neq g$. Since $\Gamma$ is torsion, there is an odd power $h'$ of $h$ whose order is a power of $2$, and it is not hard to see that $g^{h'}=g^{-1}$. Since the order of $h'$ cannot be $2$ (because the unique element of order $2$ is central) or $8$ (because there are no elements of order $8$), every non-central element of $\Gamma$ is inverted by an element of order $4$.

        Since $\Gamma$ has more than $2$ elements, it is not equal to its centre, so it follows that there is an element of order $4$.

        In particular, there are two elements $g,h$ of order $4$ such that $g^h=g^{-1}$ and $g^2=h^2$. It follows that $G=\langle g,h\rangle$ is a group of order $8$ which is not cyclic and has a unique element of order $2$. It follows (from classification of groups of order $8$) that $G\cong Q_8$. But $Q_8$ has an automorphism of order $4$. By inner ultrahomogeneity, this is realised by conjugation by some $\gamma\in \Gamma$. We may assume without loss of generality that the order of $\gamma$ is a power of $2$ (by replacing it by $\gamma^n$ for an appropriate $n\equiv 1\pmod 4$), and clearly, $\gamma^2$ is not central. It follows that the order of $\gamma$ is at least $8$, a contradiction.
    \end{proof}

    Note that the torsion hypothesis is essential: Example~\ref{ex:nontrivial_centre} shows that a non-torsion inner ultrahomogeneous group can have a nontrivial centre.

    \subsection{Divisibility}

    \begin{proposition}
        \label{prop:char_divisibility}
        If $G$ is an ultrahomogeneous group, $n$ is a positive integer and $g$ has order $m$, then $g$ is an $n$-th power if and only if $G$ has an element of order $nm$ (where $n\infty=\infty$).
    \end{proposition}
    \begin{proof}
        Let $h$ be an element of order $nm$. Then $h^n$ has order $m$, so it is conjugate to $g$ via an automorphism. If $k$ is the image of $h$ via this automorphism, then $k^n=g$.
    \end{proof}

    Observe that, as noted in \hyperref[mainthm]{Main Theorem}, Proposition~\ref{prop:char_divisibility} easily implies that if $\Gamma$ is torsion free, then it is divisible, and otherwise, it is divisible if and only if it has elements of all finite orders.

    \subsection{Normal subgroups}
    In this section, we will give some sufficient conditions for (uniform) simplicity of an inner ultrahomogeneous groups, and also discuss some particular normal subgroups other than the centre.

    \begin{remark}
        If $\Gamma$ is inner ultrahomogeneous, then all normal subgroups of $\Gamma$ are characteristic. In particular, $\Gamma$ is simple if and only if it is characteristically simple.
    \end{remark}

    Per Remark~\ref{rem:conj_order}, for any $n\in \bN\cup\{\infty\}$, the elements of order $n$ form a single conjugacy class, and so they generate normal subgroup. Within this section, we will call them $\Gamma_{n}$.

    Furthermore, for each prime $p$, we will write $\Gamma_{p^\infty}$ for the union $\bigcup_n \Gamma_{p^n}$. (Note that this is a normal subgroup, for if $\Gamma_{p^{n+1}}$ is nontrivial, then it contains $\Gamma_{p^n}$.) Finally, let us write $\Gamma_{<\infty}$ for the subgroup generated by all elements of finite order.

    \begin{remark}
        \label{rem:simpicity_criterion}
        An inner ultrahomogeneous $\Gamma$ is simple if and only if for each $p$ prime or $\infty$, the group $\Gamma_p$ is trivial or equal to $\Gamma$.
    \end{remark}

    \begin{remark}
        Note that we always have $\Gamma=\Gamma_\infty$ or $\Gamma=\Gamma_{<\infty}$ (if $\Gamma\neq \Gamma_{<\infty}$, then $\Gamma_\infty$ contains its complement, and the complement of any proper subgroup always generates the whole group).
    \end{remark}

    \begin{remark}
        If $\Age(\Gamma)$ is closed under $\times\bZ$, then $\Gamma_\infty=\Gamma$ (because every element is a product of two elements of infinite order).
    \end{remark}

    \begin{proposition}
        \label{prop:free_prod_simple}
        If $\Age(\Gamma)$ is closed under $*\bZ$ (in particular, by Proposition~\ref{prop:implications_between_props}, if it is closed under finitary HNN-extensions), and $\Gamma$ is inner ultrahomogeneous, then it is uniformly simple.
    \end{proposition}
    \begin{proof}
        Let $a\in \Gamma$ is nonidentity. By hypothesis, we have a $g\in \Gamma$ of infinite order such that $\langle a,g\rangle$ is the free product of $\langle a\rangle$ and $\langle g\rangle$. Then $ag$ is of infinite order, so it follows that $a=(ag)g^{-1}$ is the product of two elements of infinite order. On the other hand, $gag^{-1}$ is of the same order as $a$, and $a(gag^{-1})$ is of infinite order. It follows,  via Remark~\ref{rem:conj_order}, that every element is a product of at most $4$ conjugates of any other element.
    \end{proof}

    \begin{proposition}
        \label{prop:finite_groups_simple}
        Suppose $\Gamma$ is a torsion inner ultrahomogeneous group into which all finite groups embed. Then $\Gamma$ is uniformly simple.
    \end{proposition}
    \begin{proof}
        Since any permutation is the product of a most two involutions, and for every $n>2$, there are two $n$-cycles in $S_{n+1}$ whose product is a pair of disjoint transpositions: $(1,2,\ldots,n)(n,n+1,n-1,\ldots,2)=(1,2)(n,n+1)$, it is clear from the hypothesis (and Remark~\ref{rem:conj_order}) that for every $n$, the product of at most $4$ elements of order $n$ can have any finite order and $\Gamma$ is uniformly simple.
    \end{proof}

    \begin{proposition}
        \label{prop:2_or_inf_generate}
        If $\Gamma$ is inner ultrahomogeneous, then $\Gamma=\Gamma_\infty\cdot \Gamma_{2^\infty}$. In particular, if $\Gamma$ is torsion, then $\Gamma=\Gamma_{2^\infty}$, and if $\Gamma$ has no elements of order $2$, then $\Gamma=\Gamma_\infty$.
    \end{proposition}
    \begin{proof}
        Fix any $g\in \Gamma$. If the order of $g$ is infinite or a power of $2$, then clearly $g\in \Gamma_\infty\cdot \Gamma_{2^\infty}$. Consider the case when the order of $g$ is finite and odd. Then there is some $h\in \Gamma$ such that $g^h=g^{-1}$. By raising $h$ to an odd power, we can ensure that $h\in \Gamma_\infty$ or $h\in \Gamma_{2^\infty}$.

        But then
        $1=g^hg=h^{-1}ghg=h^{-1}h^{g^{-1}}g^2$, so $g^{-2}=h^{g^{-1}}h^{-1}\in \Gamma_\infty\cdot \Gamma_{2^\infty}$ (because $\Gamma_\infty\cdot \Gamma_{2^\infty}$ is normal). Since the order of $g$ is odd, it follows that $g\in \Gamma_\infty\cdot \Gamma_{2^\infty}$.

        Finally, if $g\in\Gamma$ is of order $2^nm$, where $m$ is odd, then clearly $g\in \langle g^m,g^{2n}\rangle$ which is contained in $\Gamma_\infty\cdot \Gamma_{2^\infty}$, since $g^m\in \Gamma_{2^\infty}$ and the order of $g^{2^n}$ is odd.
    \end{proof}

    \begin{corollary}
        \label{cor:p_powers_generate_gamma}
        If $\Gamma$ is torsion inner ultrahomogeneous, not of finite exponent, then for each prime $p$, $\Gamma=\Gamma_{p^\infty}$.
    \end{corollary}
    \begin{proof}
        Fix any $p$ and $n$. We will find an element $g$ of $p$-power order such that a product of its conjugates is of order $n$, which will complete the proof by inner ultrahomogeneity.

        Let $N=\max(2n,p,5)$. Note that by Theorem~\ref{thm:char_of_finite_exp}, $\Gamma$ has a subgroup $A$ isomorphic to $(\bZ/2\bZ)^N$, generated by $g_1,\ldots,g_N$.

        Note that by inner ultrahomogeneity, for any $\sigma\in S_N$, there is a $g_\sigma$ such that for $k=1,\ldots,N$ we have $g_k^{g_\sigma}=g_{\sigma(k)}$. In particular, for any $\tau,\sigma\in S_N$, we have
        \[
            g_k^{g_\sigma^{g_\tau}}=(g_k)^{g_\tau^{-1} g_\sigma g_\tau}=g_{\tau^{-1}(k)}^{g_\sigma g_\tau}=g_{\tau\sigma\tau^{-1}(k)}=g_{\sigma^{\tau^{-1}}(k)}.
        \]

        Let $g=g_{(1,\ldots,p)}$. We may assume without loss of generality that the order of $g$ is a $p$-power (by raising it to an appropriate power if necessary).

        Since $N\geq 5$, all nontrivial normal subgroups of $S_N$ contain $A_N$. In particular, the normal subgroup generated by $(1,\ldots,p)$ contains $A_N$. It follows that if $\sigma\in A_N$, then we can choose $g_\sigma$ in the group generated by $g^{\Gamma}$.

        In particular, this holds for $\sigma=(1,2,\ldots,n)(n+1,\ldots,2n)$. Clearly, the order of $g_\sigma$ is divisible by $n$ in this case, so a power of $g_\sigma$ is of order $n$, which completes the proof.
    \end{proof}

    \begin{remark}
        If $\Age(\Gamma)$ is closed under $\times G$ for some fixed nontrivial $G$, then the conclusion of Corollary~\ref{cor:p_powers_generate_gamma} can be improved to say that for each $p$, there is a \emph{fixed} $k$ such that $\Gamma=\Gamma_{p^k}$.
    \end{remark}

    \begin{proposition}
        If $\Gamma$ is inner ultrahomogeneous, then it is not a nontrivial direct product.
    \end{proposition}
    \begin{proof}
        We will argue by contraposition: suppose towards contradiction that  $\Gamma=G_1\times G_2$ is ultrahomogeneous with nontrivial factors. If $G_1$ has an element $g$ of infinite order and $h\in G_2$ is nonidentity, then $gh\in G_1h$ is of infinite order, so it is conjugate to $g$, a contradiction, since $G_1\unlhd \Gamma$. Thus, $G_1$ is torsion, and similarly $G_2$.

        Now, it is not hard to see that $G_1$ and $G_2$ must also be inner ultrahomogeneous, so by Proposition~\ref{prop:2_or_inf_generate}, if they are nontrivial, each has an element of order $2$. But by Remark~\ref{rem:conj_order}, this contradicts their normality in $\Gamma$.
    \end{proof}

    \begin{proposition}
        If $\Gamma$ is an inner ultrahomogeneous group, then its commutator subgroup is the group generated by the set of squares. In particular, if $\Gamma$ is $2$-divisible, then it is equal to $\Gamma$. In any case, the index of the commutator subgroup is at most $2$.
    \end{proposition}
    \begin{proof}
        By inner ultrahomogeneity, for every $g$ there is some $h$ such that $h^{-1}gh=g^h=g^{-1}$. In particular, $[g,h^{-1}]=g^2$.

        The opposite inclusion is true for any group: the subgroup of $G$ generated by squares is normal, and the quotient is a group of exponent $2$, hence abelian.

        The fact that the commutator subgroup is of index at most $2$ follows from the first part and Proposition~\ref{prop:char_divisibility}, since they imply that all nontrivial cosets contain an element of order $2^n$ (for a fixed $n$), so nonidentity elements in the abelianisation are conjugate.
    \end{proof}

    (Note that $\Gamma=S_3$ shows that the derived subgroup can be a proper subgroup of $\Gamma$.)

    \subsection{Disjoint amalgamation and the definability of subgroups}
    Recall the following definition.
    \begin{definition}
        Let $\cK$ be a class of first order structures. We say that $\cK$ has \emph{disjoint amalgamation} if for every diagram of the shape $A\leftarrow C\rightarrow B$ in $\cK$, there is some $D\in \cK$ into which $A,B$ embed and such the images of $C$ via both embeddings are exactly the intersection of images of $A$ and $B$.
    \end{definition}
    For example, $\cK$ has disjoint amalgamation if either it is either closed under amalgamated free products, or it is the class of finite groups (by Fact~\ref{fct:finite_amalg}).
    \begin{remark}
        If $\cK$ has the amalgamation property, then in the above, it suffices to consider the case when $A=B$ and the two embeddings of $C$ coincide.
    \end{remark}
    \begin{fact}
        \label{fct:disjoint_non_algebraicity}
        If $M$ is an ultrahomogeneous structure and $\Age(M)$ has disjoint amalgamation, then for every finitely generated substructure $A\subseteq M$ and element $b\in M\setminus A$, the orbit $\Aut(M/A)\cdot b$ is infinite.
    \end{fact}
    \begin{proof}
        This is well-known and the proof is standard.
    \end{proof}
    The following proposition shows that in a quite general context, we can actually quite easily define the subgroup of an inner ultrahomogeneous group generated by a given finite set.
    \begin{proposition}
        \label{prop:subgroup_definability}
        If $\Gamma$ is inner ultrahomogeneous and $\Age(\Gamma)$ has disjoint amalgamation (or, more generally, $\Gamma$ satisfies the conclusion of Fact~\ref{fct:disjoint_non_algebraicity}), then for every finite $A_0\subseteq \Gamma$ we have $C^2(A_0)=\langle A_0\rangle$; in particular, $\langle A_0\rangle$ is a definable subset of $\Gamma$ and for each $n$, the family of $n$-generated subgroups of $\Gamma$ is uniformly definable.
    \end{proposition}
    \begin{proof}
        Immediately follows from Fact~\ref{fct:disjoint_non_algebraicity} and Remark~\ref{rem:centr_orbit}.
    \end{proof}

    \begin{remark}
        Note that the conclusion of Proposition~\ref{prop:subgroup_definability} is false if $\Gamma$ is nontrivial centre and $\langle A_0\rangle$ does not contain the centre (clearly, $Z(\Gamma)\leq C^2(A_0)$). For instance, this happens for $A_0$ generating a torsion-free subgroup of the group in Example~\ref{ex:nontrivial_centre}.
    \end{remark}

    \begin{remark}
        On the other hand, at least in some special cases, variants of Proposition~\ref{prop:subgroup_definability} do hold without assuming disjoint amalgamation, possibly with different formulas. See e.g.\ Proposition~\ref{prop:odd_cycle_definability}.
    \end{remark}

    \section{Model-theoretic properties}
    \label{sec:mt}
    \subsection{Preliminary observations}
    In this section, we make some preliminary observations of mostly model-theoretic nature which will be useful throughout.

    We start by noting an important consequence of the trichotomy in Corollary~\ref{cor:trichotomy}.
    \begin{lemma}
        \label{lem:inf_exp_indiscernible}
        If $\Gamma$ is an inner ultrahomogeneous, then the following are equivalent:
        \begin{itemize}
            \item
            $\Gamma$ is not of finite exponent,
            \item
            $\Gamma$ contains an infinite indiscernible set,
            \item
            for each $n$, $\Gamma$ contains an indiscernible sequence of $n$ elements.
        \end{itemize}
    \end{lemma}
    \begin{proof}
        If $\Gamma$ is not of finite exponent, then by Corollary~\ref{cor:trichotomy}, it contains a copy of either $\bZ^{\oplus \bN}$ or $(\bZ/2\bZ)^{\oplus\bN}$, and by ultrahomogeneity, their generators form an indiscernible set. The rest is straightforward.
    \end{proof}

    \begin{remark}
        \label{rem:large_partial_aut}
        If $\Gamma$ is inner ultrahomogeneous and $p$ is any partial automorphism of $\Gamma$ (not necessarily finite) and $\Gamma^*\succeq \Gamma$ is $\lvert p\rvert^+$-saturated, then there is some $g_0\in \Gamma^*$ such that for $g\in\dom p$ we have $g^{g_0}=p(g)$.
    \end{remark}

    \begin{lemma}
        \label{lem:centraliser_index}
        Suppose $\Gamma$ is inner ultrahomogeneous, not of finite exponent. Fix $n\in \bN$, $n>1$ and $g\in \Gamma$ which is either of infinite order or, if $\Gamma$ is torsion, of order divisible by $n$.
        Then $C(g)\subsetneq C(g^n)$.

        Furthermore, for each $N$, we can find $g_1,g_2,\ldots,g_N$ each of the same order as $g$, such that $\langle g,g_1,\ldots,g_N\rangle$ is naturally isomorphic to $\langle g\rangle^{N+1}$ and moreover $C(g,g_1,\ldots,g_N)\subsetneq C(g^n,g_1,\ldots,g_N)$.
    \end{lemma}
    \begin{proof}
        Consider first the case when the order of $g$ is infinite. Then $\Gamma$ is not torsion, so by Proposition~\ref{prop:non-torsion_gives_Zn}, it contains a subgroup isomorphic to $\bZ^{n+1}$. Then by inner ultrahomogeneity, there is a $h\in\Gamma$ which squares the first generator and cyclically permutes the other $n$ generators. This implies that $h$ has infinite order, so by ultrahomogeneity, we may assume that $h=g$.
        Now, let $\gamma$ be one of the cyclically permuted generators. Then $\gamma\in C(g^n)\setminus C(g)$.

        Now, suppose $\Gamma$ is torsion and $n$ divides the order of $g$. By Theorem~\ref{thm:char_of_finite_exp}, we may assume without loss of generality that $\Gamma\geq \langle g\rangle\times (\bZ/n\bZ)$. Let $g_0$ be a generator of $\bZ/n\bZ$. Then by inner ultrahomogeneity, there is a $h\in\Gamma$ such that $g^h=gg_0$ and $g_0^h=g_0$.
        It follows that $h\in C(g^n)\setminus C(g)$.

        For the ``furthermore'' part, just note that in the non-torsion case, we can instead start with $\bZ^{N+n+1}$ and have $h$ commute with the extra $N$ generators. Then $\langle h\rangle $ trivially intersects the group generated by them, because no nonidentity element of $\langle h\rangle$ commutes with the generator squared by $h$. In the torsion case, start with a group isomorphic to $\langle g\rangle^{N+1}\times (\bZ/n\bZ)$ and ensure that $h$ commutes with the generators of other copies of $\langle g\rangle$.
    \end{proof}

    \begin{remark}
        \label{rem:power_centraliser}
        The first part of Lemma~\ref{lem:centraliser_index} holds also for arbitrary inner ultrahomogeneous $\Gamma$ if $g$ is of finite order divisible by $n$ and either $n>2$ or $n=2$ and the order of $g$ is divisible by $4$.

        Furthermore, under the assumptions of the Lemma (i.e.\ if $g$ is of infinite order or $\Gamma$ is torsion of infinite exponent), one can actually show that the centralisers are not only properly contained, but $[C(g^n):C(g)]$ is infinite.
    \end{remark}

    \begin{corollary}
        \label{cor:indiscernible_strong}
        If $\Gamma$ is inner ultrahomogeneous and not of finite exponent and $\Gamma^*\succeq\Gamma$ is $\aleph_1$-saturated, then there is an infinite indiscernible set $X\subseteq\Gamma^*$ such that:
        \begin{itemize}
            \item
            $X$ freely generates an abelian subgroup of $\Gamma^*$,
            \item
            given any injective partial function $f\colon X\to X$
            there is a $g\in \Gamma^*$ such that for $x\in \dom f$ we have $x^g=f(x)$,
            \item
            for every $f\colon X\to \bN$, there is a $g\in\Gamma^*$ such that for each $x\in X$ we have $g\in C(x^{f(x)})\setminus C(x)$ if $f(x)\neq 1$ and $g\in C(x)$ otherwise.
        \end{itemize}
    \end{corollary}
    \begin{proof}
        Note that for the third part, by compactness, it is enough to consider $f$ which is equal to $1$ for all but one $x\in X$.

        If $\Gamma$ is not torsion, we can take for $X$ a subset of $\Gamma$ freely generating an abelian subgroup (this exists by Proposition~\ref{prop:non-torsion_gives_Zn}). Then the second bullet is immediate by Remark~\ref{rem:large_partial_aut}. Finally, the third bullet follows from the preceding paragraph and Lemma~\ref{lem:centraliser_index}.

        If $\Gamma$ is torsion, then by Theorem~\ref{thm:char_of_finite_exp}, for each $N$, $\Gamma$ has a subgroup isomorphic to $(\bZ/N!\bZ)^N$. Then the conclusion follows via a standard argument similar to the non-torsion case, using compactness and Lemma~\ref{lem:centraliser_index}.
    \end{proof}

    \subsection{Pseudofiniteness}
    \begin{remark}
        \label{rem:not_pseudofinite}
        If $\Gamma$ is torsion or torsion-free inner ultrahomogeneous of infinite exponent, then by Lemma~\ref{lem:centraliser_index} (or, more generally, by Remark~\ref{rem:power_centraliser}, if $\Gamma$ is e.g.\ divisible), it satisfies the following sentence:
        \[
            \forall x_1\exists x_2 C(x_2)\subsetneq C(x_1),
        \]
        which is not true in any finite group, and hence these groups are not pseudofinite. Note that written explicitly in group language, it is a $\forall\exists\forall$-sentence; on the other hand, if $\Gamma$ is locally finite, then any $\exists\forall$-sentence true in $\Gamma$ is true in a finite group.
    \end{remark}

    \subsection{Failure of \texorpdfstring{$\aleph_0$}{א0}-saturation, q.e. and smallness}

    \begin{proposition}
        \label{prop:not_small}
        If $\Gamma$ is inner ultrahomogeneous and not of finite exponent, then $\Th(\Gamma)$ is not small (in fact, $S_3(\emptyset)$ is uncountable).
    \end{proposition}
    \begin{proof}
        For each $A\subseteq \bN$, put:
        \[
            p_A(x,y,z)\coloneqq \{(x^{y^k})^z=x^{y^k}\mid k\in A\}\cup \{(x^{y^k})^z\neq x^{y^k}\mid k\notin A\}.
        \]
        It is clear that $p_A$ are pairwise inconsistent.

        On the other hand, if $\Gamma$ is inner ultrahomogeneous and it is not of finite exponent, then by Lemma~\ref{lem:inf_exp_indiscernible}, it contains an infinite indiscernible set $(a_n)_{n\in \bN}$. Then we have partial automorphisms $\sigma$ such that $\sigma(a_n)=a_{n+2}$ for each $n$ and, for each $A$, $\tau_A$ such that $\tau_A(a_{2n})=a_{2n}$ if $n\in\ A$ and $\tau_A(a_{2n})=a_{2n+1}$ if $n\notin A$. This and Remark~\ref{rem:large_partial_aut} easily imply that each $p_A$ is consistent.
    \end{proof}

    \begin{proposition}
        \label{prop:omitted_type}
        If $\Gamma$ is inner ultrahomogeneous and not of finite exponent, then the $2$-type $p(x,y)$ expressing that $C(x)\subsetneq C(y)$ and $x,y$ freely generate an abelian subgroup of $\Gamma$ is (consistent and) omitted in $\Gamma$.
    \end{proposition}
    \begin{proof}
        In more detail, $p(x,y)$ consists of a formula expressing $C(x)\subsetneq C(y)$ (which implies that $x$ and $y$ commute) as well as, for every $(k,l)\in\bZ^2\setminus\{(0,0)\}$, the formula $x^ky^l\neq 1$.

        To see that this is consistent, fix any integer $N>0$ and let $g\in \Gamma$ be of order infinite or $N^2$ if $\Gamma$ is torsion (in the torsion case, this exists by Theorem~\ref{thm:char_of_finite_exp}). Then if $-N< k,l< N$, then $g^k(g^N)^l=g^{k+Nl}$. Then $\abs{k+Nl}\leq \abs{k}+N\abs{l}<N-1+N(N-1)=N^2-1<N^2$, and $k+Nl\neq 0$, so $g^k(g^N)^l\neq 1$, and by Lemma~\ref{lem:centraliser_index}, $C(g)\subsetneq C(g^N)$.

        To see that this type is omitted, just note that if $g_1,g_2$ freely generate an abelian subgroup of $\Gamma$, then there is a partial automorphism which fixes $g_1$ and squares $g_2$. The corresponding witness to inner ultrahomogeneity is in $C(g_1)\setminus C(g_2)$.
    \end{proof}

    \begin{remark}
        If $\Gamma$ is inner ultrahomogeneous and it has elements of all finite orders, then one can use Remark~\ref{rem:power_centraliser} to show that for each set $A$ of odd primes, the set of formulas saying that for each odd prime $p$, $C(x)=C(x^p)$ if and only if $p\in A$ is consistent. Since these sets are obviously pairwise inconsistent, this implies that $S_1(\emptyset)$ is uncountable. Since inner ultrahomogeneity easily implies that there are only countably many $1$-types realised in $\Gamma$ (even if $\Gamma$ is uncountable), we see that in such case, most of the types in $S_1(\emptyset)$ are necessarily omitted.
    \end{remark}

    \begin{proposition}
        \label{prop:no_qe}
        If $\Gamma$ is inner ultrahomogeneous and not of finite exponent, then the formula $C(x)\subseteq C(y)$ is not equivalent to any quantifier-free formula (in the language of groups); in particular, $\Gamma$ does not admit elimination of quantifiers.
    \end{proposition}
    \begin{proof}
        By Proposition~\ref{prop:omitted_type}, we have in a saturated model two elements $a,b$ such that $(a,b)$ and $(b,a)$ have the same quantifier-free type, but $C(a)\subseteq C(b)\not\subseteq C(a)$.
    \end{proof}

    \begin{corollary}
        \label{cor:not_saturated}
        If $\Gamma$ is infinite inner ultrahomogeneous, then it does not realise all finitary types over $\emptyset$. In particular, it is not $\aleph_0$-saturated.
    \end{corollary}
    \begin{proof}
        If $\Gamma$ is of finite exponent, then by Lemma~\ref{lem:inf_exp_indiscernible}, it omits some $n$-type defining an indiscernible sequence (which is consistent, since $\Gamma$ is infinite).

        Otherwise, $\Gamma$ is not of finite exponent, then it omits a $2$-type by Proposition~\ref{prop:omitted_type}.
    \end{proof}

    \begin{corollary}
        \label{cor:not_categorical}
        If $\Gamma$ is infinite inner ultrahomogeneous, then it is not $\aleph_0$-categorical.
    \end{corollary}
    \begin{proof}
        Immediate by Corollary~\ref{cor:not_saturated}, since every $\aleph_0$-categorical structure is $\aleph_0$-saturated.
    \end{proof}

    \begin{corollary}
        Every infinite inner ultrahomogeneous group has an elementary extension which is not inner ultrahomogeneous.
    \end{corollary}
    \begin{proof}
        Immediate, since every structure is elementarily equivalent to an $\aleph_0$-saturated one.
    \end{proof}

    \begin{remark}
        \label{rem:elem_subst}
        It is easy to see that an elementary substructure of an inner ultrahomogeneous group is also inner ultrahomogeneous.
    \end{remark}

    \begin{remark}
        If $\Gamma$ is inner ultrahomogeneous and $\Gamma^*\equiv \Gamma$, then for any finite tuples $a,b\in\Gamma^*$, if $\qftp(a)=\qftp(b)$ is isolated, then they are conjugate in $\Gamma^*$
    \end{remark}

    \begin{corollary}
        If $\Gamma$ is inner ultrahomogeneous and $\Gamma^*$ is an atomic model of $\Th(\Gamma)$, then $\Gamma^*$ is inner ultrahomogeneous.
    \end{corollary}
    \begin{proof}
        Immediate by the preceding remark.
    \end{proof}

    \begin{remark}
        It is possible for an inner ultrahomogeneous group to be atomic, for instance, this is true if it is locally finite.
    \end{remark}

    \subsection{Untameness}

    \begin{theorem}
        \label{thm:inf_exp_untame}
        Let $\Gamma$ be an inner ultrahomogeneous group of infinite exponent. Then:
        \begin{itemize}
            \item
            the formula $xy=yx$ has the independence property (in particular, it has the order property),
            \item
            for each $n$, the formula $\varphi(\bar x;y_1,y_2)=y_1^{x_1x_2\cdots x_n}\in C(y_2)$ has the $n$-independence property,
            \item
            the formula $C(x)\subseteq C(y)$ has the strict order property,
            \item
            the formula $x\in C(y_2)\setminus C(y_1)$ has the tree property of the second kind.
        \end{itemize}
    \end{theorem}
    \begin{proof}
        Recall that by Lemma~\ref{lem:inf_exp_indiscernible}, not having finite exponent is equivalent to having an infinite indiscernible set $X$ in $\Gamma$. Let $(a_n)_{n\in\bN}$ enumerate one. Fix a $\lvert\Gamma\rvert^+$ saturated $\Gamma^*\succeq\Gamma$.

        Then by Remark~\ref{rem:large_partial_aut}, every injective partial function $X\to X$ is extended by conjugation by some $\gamma\in\Gamma^*$.

        In particular, for each $A\subseteq \bN$, we can take $\gamma_A\in\Gamma^*$ such that $a_{2n}^{\gamma_A}=a_{2n}$ for $n\in A$ and $a_{2n}^{\gamma_A}=a_{2n+1}(\neq a_{2n})$ for $n\notin A$, thus showing IP for the formula $xy=yx$.

        For IP$_n$, let $(a_{v,i})_{v\in \bZ^{\oplus\bN},i\in\{0,1\}}$ enumerate an indiscernible set in $\Gamma$, as per Lemma~\ref{lem:inf_exp_indiscernible}. Via Remark~\ref{rem:large_partial_aut}, for each $A\subseteq \bZ^{\oplus\bN}$, let $g_A\in\Gamma^*$ be such that $a_{v,0}^{g_A}$ equals $a_{v,0}$ if $v\in A$, and $a_{v,1}$ otherwise, and for each $w\in \bZ^{\oplus\bN}$, let $g_w\in\Gamma^*$ be such that $a_{v,0}^{g_w}=a_{v+w,0}$.

        For each $i,j\in \bN^2$, let $g_{i,j}=g_{je_i}$, where $e_i$ is the $i$-th vector in the standard basis of $\bZ^{\oplus\bN}$.
        Then for any $i_0,\ldots,i_n\in\bN$ we have $a_{0,0}^{g_{0,i_0}g_{1,i_1}\cdots g_{n,i_n}}=a_{(i_0,i_1,\ldots,i_n,0,0,\ldots),0}$, and hence for $A\subseteq \bZ^{\oplus\bN}$ we have $a_{0,0}^{g_{0,i_0}g_{1,i_1}\cdots g_{n,i_n}}\in C(g_A)$ if and only if $(i_0,i_1\ldots,i_n,0,\ldots)\in A$. The conclusion follows.

        To see that $C(x)\subseteq C(y)$ has the strict order property, just note that we have a type
        \[
            p(x)=\{C(x^{2^n})\subsetneq C(x^{2^{n+1}})\mid n\in \bN\}
        \]
        This is consistent by Lemma~\ref{lem:centraliser_index}: if $\Gamma$ is not torsion, then this is simply realised by any element of $\Gamma$ infinite order. If $\Gamma$ is torsion, then any finite part is realised by an element of order $2^n$ for a sufficiently large $n$.

        Finally, for TP$_2$, fix a set as in Corollary~\ref{cor:indiscernible_strong}, enumerate it as $(g_n)_{n\in\bN}$. Let $(p_n)_{n\in\bN}$ be an enumeration of prime numbers. Then TP$_2$ for $x\in C(y_2)\setminus C(y_1)$ is witnessed by the array $((g_n,g_n^{p_m})_{n,m}$. Indeed, by Corollary~\ref{cor:indiscernible_strong}, for each $f\colon \bN\to \bN$ there is a $g$ such that for all $n$ we have $g\in C(g_n^{p_f(n)})\setminus C(g_n)$, so every path is consistent, but if $p,q$ are distinct primes, then $C(g_n^p)\cap C(g_n^q)=C(g_n)$, so every row is $2$-inconsistent.
    \end{proof}

    \begin{remark}
        In the language of groups, the formulas witnessing IP and IP$_n$ are positive quantifier-free, the formula witnessing TP$_2$ is quantifier-free, but not positive, while the formula witnessing SOP is universal and not positive (and not equivalent to any quantifier-free formula by Proposition~\ref{prop:no_qe}). By \cite[Theorem 4.2]{SU06}, SOP (or even SOP$_4$) cannot be witnessed by a quantifier-free formula, or even a quantifier-free type.
    \end{remark}

    \subsection{Non-rosiness}
    One might think that due to the existence of a stationary independence relation in some inner ultrahomogeneous groups (see Proposition~\ref{prop:free_stationary_indep}), there is a hope that they might be rosy (see \cite{CKP08} for the definition and more details). However, this is not true, at least for the groups of infinite exponent.

    \begin{corollary}
        If $\Gamma$ is inner ultrahomogeneous and not of finite exponent, then it is not rosy.
    \end{corollary}
    \begin{proof}
        By Corollary~\ref{cor:indiscernible_strong} (or Lemma~\ref{lem:centraliser_index}), an inner ultrahomogeneous group does not satisfy the ucc condition of \cite[Proposition 1.3]{CKP08}. Hence, it is not rosy.
    \end{proof}

    \subsection{Straight maximality}
    Recall the following property of a first order theory, defined in e.g.\ \cite{Coo82} (as $[1]$) and in \cite[Definition 5.20]{Sh00}.
    \begin{definition}
        We say that a complete first order theory $T$ is \emph{straightly maximal} if here is a formula $\varphi(\bar x,\bar y)$ such that for every quantifier-free formula $\chi(z_1,\ldots,z_n)$ in the language of Boolean algebras and every (equivalently, some) $M\models T$, we can find $\bar b_1,\ldots, \bar b_n\in M$ such that $\mathcal P(M^{\bar x})\models \chi(\varphi(M,\bar b_1),\ldots,\varphi(M,\bar b_n))$.
    \end{definition}
    Note that straight maximality implies e.g.\ TP$_2$ and the strict order property, and true to its name, it is the strongest in a certain class of ``wildness'' conditions (it does not, however, seem to imply the $n$-independence property).

    We will use the following simple criterion for straight maximality.

    \begin{fact}
        \label{fct:str_max_char}
        $T$ is straightly maximal if and only if there is a formula $\varphi(\bar x,\bar y)$ such that for every $n\in\bN$, there are $n$ disjoint nonempty $\varphi$-definable sets such that the union of any number of them is $\varphi$-definable.
    \end{fact}
    \begin{proof}
        This is \cite[Lemma 3, Section 2.0]{Coo82} (and not hard to prove directly).
    \end{proof}

    \begin{lemma}
        \label{lem:odd_abelian_fix_elt}
        If $G$ is a finite abelian group of odd order and $g\in G$ does not generate $G$, then there is a nontrivial automorphism of $G$ which fixes $g$.
    \end{lemma}
    \begin{proof}
        If the order of $g$ equals the exponent of $G$, then $\langle g\rangle$ is a direct factor. Since it is not equal to $G$, we have an automorphism which fixes it and inverts the group complementing it (which is of odd order greater than $1$, so this is nontrivial).

        Otherwise, if we decompose $G$ into direct sum of cyclic groups of prime order, the projection of $\langle g\rangle$ into one of them, say $G_1$, is not onto.

        Write $G=G_1\times G_2$ and $g=g_1+g_2$ where $g_1\in G_1$, $g_2\in G_2$. Then $g_1$ does not generate $G_1$ and $G_1$ is a $p$-group for an odd prime $p$.

        Let $d$ be the order of $g_1$. Then $d=1$ or $p$ divides $d$, so $p$ does not divide $d+1$. It follows that $G_1$ has a nontrivial automorphism given by $x\mapsto x^{d+1}$ which fixes $g_1$, which we can extend by identity on $G_2$, yielding a nontrivial automorphism fixing $g$.
    \end{proof}

    \begin{proposition}
        \label{prop:odd_cycle_definability}
        Suppose $\Gamma$ is inner ultrahomogeneous and $g\in\Gamma$ is of finite odd order. Then
        \[
            \langle g\rangle = \{h^2\mid h\in C^2(g)\}
        \]
    \end{proposition}
    \begin{proof}
        $\subseteq$ is clear since the order of $\langle g\rangle$ is odd, so squaring is bijective. For $\supseteq$, let $h\in C^2(g)$.

        Note that this implies that any finite partial automorphism of $\langle h,g\rangle$ which fixes $g$ must fix its entire domain. Furthermore, this implies also that for every $k\in\bZ$ we have $h^k\in C^2(g)$. Write $n$ for the order of $h$.

        If $n=\infty$, then $\langle g,h\rangle\cong \langle g\rangle\times \langle h\rangle$, so $\langle g,h\rangle$ has a nontrivial automorphism which fixes $g$ and inverts $h$, which is a contradiction.

        If $n$ is divisible by $4$, then similarly, $g$ generates a direct factor of $\langle g,h^{n/4}\rangle$, and inverting $h^{n/4}$ yields a nontrivial automorphism fixing $g$.

        Thus, the order of $h^2$ is finite and odd. But then $\langle g,h^2\rangle$ is abelian of odd order. Since every automorphism fixing $g$ is trivial, by Lemma~\ref{lem:odd_abelian_fix_elt}, it follows that $\langle g,h^2\rangle=\langle g\rangle$, so $h^2\in\langle g\rangle$.
    \end{proof}

    \begin{corollary}
        \label{cor:prime_multiples}
        There is a formula $\varphi(x,y)$ such that for any $g\in \Gamma$ of finite odd order we have $\varphi(g,h)$ if and only if $h\in \langle g\rangle$ and the order of $h$ is prime.
    \end{corollary}
    \begin{proof}
        Let $\psi(x,y)$ be a formula expressing that $y$ is the square of an element of $C^2(x)$ (so that for $g\in\Gamma$ of odd order we have, per  Proposition~\ref{prop:odd_cycle_definability},  $\psi(g,\Gamma)=\langle g\rangle$). Let $\chi(x)$ be the formula saying that $x\neq 1\land \forall y(\psi(x,y)\land y\neq 1\rightarrow \psi(y,x))$, so that for $g\in\Gamma$ of odd order, $\Gamma\models\chi(g)$ if and only if the order of $g$ is prime.
        Then it is not hard to see that $\varphi(x,y)=\psi(x,y)\land \chi(y)$ works.
    \end{proof}

    \begin{theorem}
        \label{thm:straight_maximality}
        Suppose $\Gamma$ is inner ultrahomogeneous and for each $n$, there is a $g\in \Gamma$ of finite order with at least $n$ distinct prime divisors. Then $\Gamma$ is straightly maximal.
    \end{theorem}
    \begin{proof}
        Let $\varphi(x,y)$ be the formula as in Corollary~\ref{cor:prime_multiples}.
        Note that the hypothesis of the theorem implies that for each $n$, we can find a set $P$ of $n$ odd primes and an element $g\in\Gamma$ of order $\prod P$. Clearly, for any $A\subseteq P$ if we put $g_A=g^{\prod (P\setminus A)}$, then the order of $g_A$ is $\prod A$, and $\varphi(g_A,\Gamma)=\{h\in \langle g\rangle\mid \textrm{the order of }h\textrm{ is in }A\}$. The conclusion follows by Fact~\ref{fct:str_max_char}.
    \end{proof}

    It may be interesting to ask whether the hypotheses of Theorem~\ref{thm:straight_maximality} are necessary. For example:
    \begin{question}
        Suppose $\Gamma$ is infinite torsion-free inner ultrahomogeneous. Can $\Gamma$ be straightly maximal?
    \end{question}

    \section{Ample generic automorphisms}

    In this section, we provide simple sufficient conditions for the ampleness of generic automorphisms of an inner ultrahomogeneous group, generalising known results about Hall's universal group.

    \begin{definition}
        Fix a topological group $G$. A \emph{generic $n$-tuple} in $G$ is an $n$-tuple $\bar g\in G^n$ such that $\bar g^G$ (the orbit of diagonal action of $G$ via conjugation) is comeagre in $G^n$.
        We say that $G$ has \emph{ample generics} if for each $n\in \bN$, it has a generic $n$-tuple.

        If $M$ is a first order structure, then we say that $M$ has \emph{ample generic automorphisms} if $\Aut(M)$ (with the usual group topology, given by stabilisers of finite sets) has ample generics.
    \end{definition}

    The main tool in the proof will be the following lemma.
    \begin{lemma}
        \label{lem:eppa_amalgam}
        Suppose $\cK$ is a hereditary class of groups with inner EPPA and AP. Suppose furthermore that $\cK$ is closed under $\times$ or under $\times\bZ$.

        Let $A\leq B,C$ be $\cK$-groups. Fix partial automorphisms $p_1,p_2,\ldots,p_n$ of $B$ and $q_1,\ldots,q_n$ of $C$, such that for each $j$ we have $p_j\restr_{A}=q_j\restr_A\in \Aut(A)$.

        Then there is an amalgam $D\in \cK$ of $B,C$ over $A$ and elements $g_1,\ldots,g_n\in D$ such that (identifying $B,C$ with subgroups of $D$ containing $A$) for each $b\in B, c\in C$ and $j$ we have $b^{g_j}=p_j(b)$ and $c^{g_j}=q_j(c)$.
    \end{lemma}
    \begin{proof}
        Let us first prove the following Claim:
        \begin{clm}
            If $\cK$ is closed under $\times\bZ$, then for each $G\in \cK$ and $n$-tuple $p_1,\ldots, p_n$ of partial automorphisms of $G$, we can find $H\geq G$ in $\cK$ and $h_1,\ldots,h_n\in H$ which witness inner EPPA for $p_1,\ldots,p_n$ and freely generate a subgroup of $H$ which intersects $G$ trivially.
        \end{clm}
        \begin{clmproof}
            By hypothesis, $G\times \bZ^2\in\cK$. Fix $f_1,\ldots,f_n\in \Aut(\bZ^2)$ which freely generate a subgroup of $\Aut(\bZ^2)$. Then for each $k$, $p_k\cup f_k$ is a finite partial automorphism of $G\times \bZ$. If we take $H\geq G\times \bZ$ with witnesses $h_1,\ldots,h_n\in H$ of inner EPPA for $p_1\cup f_1,\ldots,p_n\cup f_n$, then it is easy to see that they work.
        \end{clmproof}

        \begin{clm}
            For any $A,B,C,p_k,q_k$ as in the statement of the theorem, we can find groups $\bar B\geq B$, $\bar C\geq C$, and for $k=1,\ldots, n$ some $b_k\in \bar B$, $c_k\in \bar C$ which are witnesses of inner EPPA for $p_k$ and $q_k$, respectively, and such that
            \[
                \langle A,b_1,\ldots, b_n\rangle\cong \langle A,c_1,\ldots, c_n\rangle
            \]
            via an isomorphism fixing $A$ and sending $b_k$ to $c_k$ for $k=1,\ldots,n$.
        \end{clm}
        \begin{clmproof}
            If $\cK$ is closed under $\times\bZ$, then first, apply the preceding claim to $G=B$ and $p_1,\ldots,p_n$, and put $\bar B=H$, $b_k=h_k$ for $k=1,\ldots, n$. Then, since $\langle b_1,\ldots,b_n\rangle$ normalizes $A$ and intersects it trivially, $\langle A,b_1,\ldots,b_n\rangle\cong A\rtimes F_n$ (where $F_n$ is the free group of rank $n$), where the action is given by $a^{b_k}=p_k(a)$.

            Since we can repeat the same procedure for $C$, and for $a\in A$ we have $p_k(a)=q_k(a)$, the conclusion follows.

            Now, if $\cK$ is closed under $\times$, consider $B\times C$. Put $A'=A\times \{1\}$, $A''=\{1\}\times A$. Then for each $k$, $p_k\times q_k$ is a partial automorphism of $B\times C$. Let $G\geq B\times C$ and $g_1,\ldots,g_n\in G$ be the corresponding witnesses of inner EPPA. Put $H=G\times \langle g_1,\ldots, g_n\rangle$.

            Finally, put $\bar B=\langle B\times \{1\}^2, (g_1,g_1),\ldots, (g_n,g_n)\rangle\leq H$ and similarly, $\bar C=\langle \{1\}\times C\times \{1\}, (g_1,g_1),\ldots(g_n,g_n)\rangle\leq H$ (identified with supergroups of $B$ and $C$ respectively), and for $k=1,\ldots,n$ put $b_k=c_k=(g_k,g_k)$. Again, the group generated by $(g_k,g_k)$, $k=1,\ldots, k$ normalises and intersects trivially both $A'\times \{1\}$ and $A''\times \{1\}$ (i.e. the copies of $A$ in $\bar B$ and $\bar C$, respectively), we conclude as in first two paragraphs.
        \end{clmproof}
        Finally, to prove the lemma, take $\bar B$ and $\bar C$ as in Claim 2. Then if we take for $D$ their amalgam along $\langle A,b_1,\ldots, b_n\rangle\cong \langle A,c_1,\ldots, c_n\rangle$, the conclusion is clearly satisfied.
    \end{proof}

    The following proposition basically says that inner EPPA implies ($n$-)EPPA.
    \begin{proposition}
        \label{prop:n-eppa}
        If $\Gamma$ is inner ultrahomogeneous, $A\leq \Gamma$ is finitely generated and $p_1,\ldots,p_n$ is an $n$-tuple of partial automorphisms of $A$, then there is $B\leq \Gamma$ containing $A$ and $\sigma_1,\ldots,\sigma_n\in \Aut(B)$ such that $\sigma_i$ extends $p_i$ for $i=1,\ldots,n$.
    \end{proposition}
    \begin{proof}
        Let $g_1,\ldots,g_n$ be witnesses for inner ultrahomogeneity of $\Gamma$ corresponding to $p_1,\ldots,p_n$ and take $B=\langle B,g_1,\ldots,g_n\rangle$ and let $\sigma_i$ be the inner automorphism given by $g_i$ for $i=1,\ldots,n$.
    \end{proof}

    We will use the following standard criterion for the existence of generic automorphisms.

    \begin{fact}
        \label{fct:cap_ample_gen}
        Let $\cK$ be a Fraïssé class and $M$ its limit.

        Fix any $n\in \bN$. Let $\cK_p^{(n)}$ be the class of $\cK$-structures with $n$-tuples of partial automorphisms.

        Then if $\cK_p^{(n)}$ has the joint embedding property and a cofinal subclass of amalgamation bases, then there is a generic $n$-tuple in $\Aut(M)$. In particular, if this is true for all $n$, then $M$ has ample generic automorphisms.
    \end{fact}
    \begin{proof}
        This follows from \cite[Theorem 1.2]{Ivanov} or \cite[Theorem 6.2]{KR07}. The special case of $n=1$ is \cite[Theorem 2.1]{Tru92}.
    \end{proof}

    \begin{theorem}
        \label{thm:ample_generics}
        Suppose $\Gamma$ is a countable inner ultrahomogeneous group and $\Age(\Gamma)$ is closed under $\times$ or $\times\bZ$. Then $\Gamma$ has ample generic automorphisms.
    \end{theorem}
    \begin{proof}
        Let $\cK=\Age(M)$, so that $M$ is the Fraïssé limit of $\cK$ (since it is countable and ultrahomogeneous). By Fact~\ref{fct:cap_ample_gen}, it suffices to show that for each $n\in\bN$, $\cK_p^{(n)}$ has JEP and a cofinal subclass of amalgamation bases. Note that Lemma~\ref{lem:eppa_amalgam} implies that the class of $\cK_t^{(n)}$ of $\cK$-groups with $n$-tuples of (total) automorphisms consists of amalgamation bases. This class is cofinal by by Proposition~\ref{prop:n-eppa}. JEP of $\cK_p^{(n)}$ follows, since every partial automorphism can be extended to one with the identity element in its domain (which is then necessarily fixed), and the trivial group with an $n$-tuple of trivial automorphisms is in $\cK_t^{(n)}$.
    \end{proof}

    \begin{remark}
        If $\Gamma$ is torsion, then in order to obtain a \emph{single} generic automorphism, it is actually enough to assume that $\Age(\Gamma)$ is closed under products with all finite cyclic groups.
    \end{remark}

    \begin{remark}
        Proposition~\ref{prop:implications_between_props} implies that Theorem~\ref{thm:ample_generics} applies in particular to any countable inner ultrahomogeneous group whose age is closed under $*\bZ$ or under finitary HNN-extensions.
    \end{remark}

    \begin{remark}
        Note that while under hypotheses of Theorem~\ref{thm:ample_generics} generic automorphisms exist, and by inner ultrahomogeneity, they can be arbitrarily closely approximated by inner automorphisms, an inner automorphism of a countable group $\Gamma$ cannot be generic (unless $\Aut(\Gamma)$ is trivial). This follows from the fact that inner automorphisms form a countable normal subgroup of $\Aut(\Gamma)$.
    \end{remark}

    \section{Examples}
    \label{sec:examples}
    \subsection{Finite groups}
    As noted in \hyperref[mainthm]{Main Theorem}, there are no finite inner ultrahomogeneous groups of more than $6$ elements. It follows that there are exactly three isomorphism classes of finite inner ultrahomogeneous group: the trivial group, the cyclic group of order $2$, and the nonabelian group of order $6$. (Incidentally, those are the first three symmetric groups.)
    \begin{example}
        The cyclic group of order $2$ is inner ultrahomogeneous and abelian (and it is the only such nontrivial group, up to isomorphism), while $S_3$ is inner ultrahomogeneous, has trivial centre, but it is not simple (its derived subgroup is a nontrivial proper normal subgroup).
    \end{example}

    \subsection{Hall's universal locally finite group}
    Let $\Gamma_F$ be Hall's universal group, i.e.\ the Fraïssé limit of the class of all finite groups. In fact, the original paper of Hall giving the first construction (which was explicit and did not use Fraïssé theory) already contains a proof that $\Gamma_F$ is inner ultrahomogeneous \cite[Lemma 3]{Hal59} and the fact that the class of finite groups has inner EPPA (Fact~\ref{fct:finite_hnn}).
    \begin{remark}
        The universal theory of $\Gamma_F$ is the universal theory of finite groups (this follows from the fact that it is locally finite and every finite group embeds into it).
    \end{remark}
    \begin{remark}
        By \cite{Slo81}, the universal theory of finite groups (= the universal theory of $\Gamma_F$) is not decidable.
    \end{remark}

    \begin{example}
        \label{ex:hall_group}
        By \hyperref[mainthm]{Main Theorem} and the above remarks, $\Gamma_F$ is uniformly simple, it has ample generic automorphisms, it is not $\aleph_0$-saturated, its theory is not decidable, does not admit q.e., is not small, is straightly maximal (so it has SOP, TP$_2$ and the order property), and it has IP$_n$ for all $n$.
    \end{example}
    (Note that, as noted in the introduction, at least some of the properties of $\Gamma_F$ noted in Example~\ref{ex:hall_group} were known, particularly the fact that it is unstable and has ample generic automorphisms.)

    \subsection{Universal locally recursively presentable group}
    Consider the class of finitely generated, recursively presentable groups.
    Corollary~\ref{cor:hnn_Fraïssé} easily implies that it is a Fraïssé class with inner EPPA.

    Thus, it has a Fraïssé limit $\Gamma_R$ which is inner ultrahomogeneous. Equivalently, we can define $\Gamma_R$ as the limit of the class of finitely presentable groups --- it is a non-hereditary class which still has inner EPPA, amalgamation and joint embedding, and its hereditary closure is the above class by Higman's embedding theorem, so the limit is the same.

    \begin{remark}
        Similarly to $\Gamma_F$, universal theory of $\Gamma_R$ is the universal theory of finitely generated recursively presentable groups. Since every such group embeds in a finitely presentable group, this is the same as the universal theory of finitely presentable groups. Undecidability of the word problem for groups implies that this universal theory is undecidable.
    \end{remark}
    \begin{remark}
        It is not hard to see that every existential statement true in some group is true in a finitely presentable group, so the universal theory of $\Gamma_R$ is in fact the universal theory of groups.
    \end{remark}

    \begin{remark}
        Since the universal theory of groups is not the universal theory of finite groups (briefly, because a finitely presentable group satisfies the latter if and only if it is residually finite), the universal theories of $\Gamma_F$ and $\Gamma_R$ do not coincide. In particular, $\Gamma_F$ and $\Gamma_R$ are not elementarily equivalent.
    \end{remark}

    We finish by noting that the following relation gives us a stationary independence relation in $\Gamma_R$ (in the sense of \cite{TZ13}).
    \begin{definition}
        Given three finite subsets $A,B,C$ of a group $G$, let us say that \emph{$A,B$ are freely independent over $C$}
        if $\langle A,B,C\rangle$ is naturally isomorphic to $\langle A,B\rangle*_{\langle B\rangle}\langle B,C\rangle$.
    \end{definition}

    \begin{proposition}
        \label{prop:free_stationary_indep}
        Suppose $G$ is an ultrahomogeneous group and $\Age(G)$ is closed under finitary amalgamated free products. Then free independence is a stationary independence relation on finite subsets of $G$.
    \end{proposition}
    \begin{proof}
        We check the axioms as listed in \cite[Definition 2.1]{TZ13}. Invariance is clear, as is symmetry. Existence follows from ultrahomogeneity and the assumption that $\Age(G)$ is closed under amalgamated free products. Stationarity follows from ultrahomogeneity.

        Monotonicity and transitivity follow from the observations that if $B\leq A,C,D\leq \Gamma$, then $A*_BC\leq A*_B\langle C,D\rangle$ and $(A*_BC)*_{C}\langle C,D\rangle\cong A*_B(C*_C\langle C, D\rangle)$.
    \end{proof}

    Note that the free independence is not stationary in saturated extensions of $\Gamma_R$, as they do not have quantifier elimination, see Proposition~\ref{prop:omitted_type} and Proposition~\ref{prop:no_qe}.

    \begin{example}
        \label{ex:lrec_group}
        By \hyperref[mainthm]{Main Theorem} and the above remarks, $\Gamma_R$ is uniformly simple, it has ample generic automorphisms, it is not $\aleph_0$-saturated, its theory is not decidable, does not admit q.e., is not small, is straightly maximal (so SOP, TP$_2$ and the order property), and it has IP$_n$ for all $n$.

        Moreover $\Gamma_R\not\equiv\Gamma_F$ and $\Gamma_R$ admits a stationary independence relation.
    \end{example}

    \subsection{Closure under HNN-extensions}
    Let us start with any group $\Gamma_0$. We recursively define an ascending sequence of groups of length $\omega$.

    Suppose we have $\Gamma_n$. Then enumerate all finite partial automorphisms $\Gamma_n\to \Gamma_n$ as $(p_i)_{i\in I}$. Then let $\Gamma_{n+1}$ be the HNN-extension $\Gamma_n*_{(p_i)_{i\in I}}$ (with a set of stable letters indexed by $I$).
    Finally, put $\Gamma=\bigcup_n\Gamma_n$. Then it is easy to check (using Corollary~\ref{cor:hnn_uh}) that $\Gamma$ is inner ultrahomogeneous of cardinality at most $\lvert \Gamma_0\rvert+\aleph_0$.

    \begin{example}
        \label{ex:any_group_in_ihg}
        The group $\Gamma$ constructed above is an inner ultrahomogeneous group containing an arbitrary group $\Gamma_0$, of cardinality $\langle \Gamma_0\rangle+\aleph_0$. By \hyperref[mainthm]{Main Theorem} it follows that e.g.\ it is it is uniformly simple, not $\aleph_0$-saturated, it has TP$_2$ and SOP, and if it is countable, then it has ample generic automorphisms. Moreover, by Proposition~\ref{prop:implications_between_props}, $\Age(\Gamma)$ is closed under finitary amalgamated free products, so by Proposition~\ref{prop:free_stationary_indep}, free independence is a stationary independence relation on $\Gamma$.

        Finally, by the torsion theorem for HNN extensions (Fact~\ref{fct:torsion_hnn}), the only finite orders of elements of $\Gamma$ are the finite orders of elements of $\Gamma_0$, so in particular, we can obtain groups which are torsion-free in this way, or groups which are not divisible (e.g.\ if $\Gamma_0$ is a nontrivial finite group).
    \end{example}

    \begin{example}
        If $\cK_0$ is any set of groups, then we can build a group $\Gamma_0$ as the direct sum of elements of $\cK_0$ and then apply the above construction. In particular, if $\cK_0$ is a countable set of finitely generated groups, the age of the group $\Gamma$ we obtain is a Fraïssé class of groups with inner EPPA containing $\cK_0$. Alternately, we can do this by simply closing $\cK_0$ under direct products and finitary HNN-extensions (and finitely generated embedded subgroups) and applying Corollary~\ref{cor:hnn_Fraïssé}.
    \end{example}

    \subsection{Existentially closed groups}
    In a similar vein to the previous example, as noted in e.g.\ \cite[Exercise 9 for Section 8.1]{Hod}, any group which is existentially closed in the class of groups is inner ultrahomogeneous. Indeed, if $\Gamma$ is existentially closed and $p\colon \Gamma\to\Gamma$ is a finite partial automorphism, then the HNN-extension $\Gamma*_p$ satisfies $\exists x (\dom(p)^x=p(\dom p))$, so the same must be true in $\Gamma$.
    Similarly, any torsion-free group which is existentially closed in the class of torsion-free groups is inner ultrahomogeneous.

    In particular, one can show that the group $\Gamma_R$ is existentially closed, as is any ultrahomogeneous group whose age is the class of finitely generated, recursively presentable groups.

    Likewise, any locally finite group which is existentially closed in the class of locally finite groups is inner ultrahomogeneous. This follows from e.g.\ the argument sketched after \cite[Definition 0.1]{Sh17} (see also \cite[Exercise 7 for Section 11.6]{Hod}). Note that in this class, there is only one countable group, namely $\Gamma_F$. Furthermore, as noted in \cite[Claim 0.14]{Sh17}, inner ultrahomogeneity and existential closedness are equivalent for a locally finite group whose age is the class of all finite groups (by Proposition~\ref{prop:char_inner_uh}, in this case, inner ultrahomogeneity is also equivalent to ultrahomogeneity).

    \begin{example}
        These groups are obviously divisible, infinite, and uniformly simple, but \hyperref[mainthm]{Main Theorem} implies easily that all these groups are also not $\aleph_0$-saturated, they have TP$_2$ and SOP (and are straightly maximal, apart from the torsion-free ones). It is not clear whether all the countable ones have ample generic automorphisms.
    \end{example}

    \subsection{An example with nontrivial centre}
    In this section, we will use a variant of the construction from Example~\ref{ex:any_group_in_ihg} to get an example of an infinite inner ultrahomogeneous group whose centre is nontrivial (so in particular, it is not simple).

    First, let us make the following observation.
    \begin{fact}
        \label{fct:no_new_finite_order_elts}
        Suppose $G$ is a group and an $n\in\bN$ is such that all elements of order $n$ are central. Suppose in addition that $p$ is a partial automorphism of $G$ such that all elements of order $n$ are in $\dom p$ and fixed by $p$. Then all elements of order $n$ in $G*_p$ are in $G$.
    \end{fact}
    \begin{proof}
        By the torsion theorem for HNN extensions (Fact~\ref{fct:torsion_hnn}), every element of order $n$ in $G*_p$ is conjugate to an element of $G$. But by hypothesis, every element of order $n$ in $G$ is central in $G*_p$. The conclusion follows.
    \end{proof}

    Consider the following construction. Let $\Gamma_0$ be cyclic of order $2$, generated by $g_0$. Then given $G_n$, we construct $\Gamma_{n+1}$ as follows. Let $P$ be the set of all isomorphisms between finitely generated subgroups of $\Gamma_n$ which contain $g_0$. Then $\Gamma_{n+1}$ is the HNN-extension $\Gamma_n*_{P}$.

    Finally, put $\Gamma=\bigcup_n \Gamma_n$; clearly, it is a countable group. Furthermore, note that by Fact~\ref{fct:no_new_finite_order_elts} and straightforward induction, $g_0$ is the only nonidentity finite order element in each $\Gamma_n$, and hence also in $\Gamma$. In particular, it is central, and if $A\leq \Gamma_n$ is a finitely generated subgroup, then either $g_0\in A$ or $\langle g_0,A\rangle\cong \langle g_0\rangle\times A$ (because $g_0$ commutes with $A$ and is not in $A$).

    It follows that every isomorphism between finitely generated subgroups of $\Gamma$ can be extended to an isomorphism between finitely generated subgroups containing $g_0$. This easily implies that $\Gamma$ is in fact inner ultrahomogeneous.

    \begin{example}
        \label{ex:nontrivial_centre}
        In summary, $\Gamma$ constructed above is a countably infinite inner ultrahomogeneous group which has nontrivial centre and an unique nonidentity element of finite order (so it is not simple, although centre is its only nontrivial proper normal subgroup). Since its age is closed under $\times\bZ$, many consequences of \hyperref[mainthm]{Main Theorem} apply to it. In particular, it embeds every countable torsion-free abelian group, it has SOP and TP$_2$ and it has ample generic automorphisms.
    \end{example}

    \begin{remark}
        Instead of a cyclic group of order $2$, we could start with any group which has a unique element of order $2$, for instance $\bQ/\bZ$. Then the resulting group will be inner ultrahomogeneous with nontrivial centre and also have elements of all finite orders. Thus, even more consequences of \hyperref[mainthm]{Main Theorem} would apply to it.
    \end{remark}

    \begin{remark}
        A similar argument can be used to show that a group which is existentially closed in the class of groups having at most one element of order $2$ is inner ultrahomogeneous with nontrivial centre.
    \end{remark}

    \subsection{Finite exponent groups?}
    The following question remains open.
    \begin{question}
        \label{qu:finite_exponent}
        Is there an infinite finite exponent inner ultrahomogeneous group? If yes, can it be stable?
    \end{question}

    By \hyperref[mainthm]{Main Theorem} we can say the following for any hypothetical infinite inner ultrahomogeneous $\Gamma$ of finite exponent: it is necessarily of exponent smaller than $(2^{100})!$, it contains no copy of $(\bZ/2\bZ)^6$, it is not $\aleph_0$-saturated, it has trivial centre and it is isomorphic to $\Aut(\Gamma)$ (which is discrete, so $\Gamma$ has no generic automorphisms).

    Moreover, it is not hard to see that inner ultrahomogeneity is inherited by elementary subgroups, so by the downwards Löwenheim-Skolem, if such $\Gamma$ exists, there is a countable one, which is then a Fraïssé limit. Thus, we may rephrase Question~\ref{qu:finite_exponent} in the following way:
    \begin{question}
        Is there a Fraïssé class of groups of uniformly bounded exponent with inner EPPA?
    \end{question}

    In fact, it is not clear whether there are infinite torsion inner ultrahomogeneous groups which are not existentially closed locally finite.

    \section*{Acknowledgements}
    I would like to thank Wiesław Kubiś for suggesting to investigate the class of finitely presentable groups. I would also like to thank Adam Bartoš, Itay Kaplan, Alex Kruckman and Krzysztof Krupiński for helpful discussions, and Martin Ziegler for pointing out that the group $\Gamma_R$ is existentially closed.
    \printbibliography
\end{document}